\documentclass[9pt]{article}
\usepackage{amsmath}
\usepackage{parskip}
\usepackage{enumerate}
\usepackage{amssymb}
\usepackage{amsthm}
\usepackage{graphicx}
\usepackage{mathabx}
\usepackage{float}
\usepackage{tikz}
\usepackage{etex}
\DeclareGraphicsExtensions{.bmp,.png,.pdf,.jpg}
\usepackage{latexsym}
\usepackage{amscd}
\usepackage[all]{xy}
\usepackage[utf8]{inputenc}
\usepackage[linkcolor=blue,citecolor=blue, colorlinks=true,backref=true,bookmarks=true]{hyperref}
\usepackage{cite}

\newtheorem{teor}{Theorem}[section]
\newtheorem{cor}[teor]{Corollary}
\newtheorem{lem}[teor]{Lemma}

\newtheorem{prop}[teor]{Proposition}
\newtheorem{defn}[teor]{Definition}
\newtheorem{rem}[teor]{Remark}

\def\R{\mathbb{R}}

\def\N{\mathbb{N}}

\numberwithin{equation}{section}

\title{\bf{Convergence of the Weighted Yamabe Flow}}

\author{Zetian Yan}

\begin{document}
\maketitle

\begin{abstract}
We introduce the weighted Yamabe flow 
\begin{displaymath}
\left\{
\begin{array}{ll}
\frac{\partial g}{\partial t} &=(r^m_{\phi}-R^m_{\phi})g\\
\frac{\partial \phi}{\partial t} &=\frac{m}{2}(R^m_{\phi}-r^m_{\phi})
\end{array}
\right.
\end{displaymath}
on a smooth metric measure space $(M^n, g, e^{-\phi}{\rm dvol}_g, m)$, where $R^m_{\phi}$ denotes the associated weighted scalar curvature, and $r^m_{\phi}$ denotes the mean value of the weighted scalar curvature. We prove long-time existence and convergence of the weighted Yamabe flow if the dimension $n$ satisfies $n\geqslant 3$.
\end{abstract}

\noindent \textbf{Key points:} Yamabe flow, Convergence, Smooth measure metric spaces. \\
\noindent\makebox[\linewidth]{\rule{\textwidth}{0.4pt}}
AMS classification (2010). Primary: 35G25. Secondary: 35K90, 47J35, 53A30.

\section{Introduction}
The Yamabe flow was first introduced by Richard Hamilton in \cite{Hamilton88}. Hamilton conjectured that, for every initial metric, the flow converges to a conformal metric of constant scalar curvature. In case $Y(M^n, g_0)\leqslant 0$, it is not difficult to show that the conformal factor is uniformly bounded above and below. Moreover, the flow convergences to a metric of constant scalar curvature as $t\to \infty$.  

The case $Y(M^n, g_0)>0$ is more interesting. Chow \cite{Chow92} proved the convergence of the flow for locally conformally flat metrics with positive Ricci curvature. Ye \cite{Ye94} later extended the result to all locally conformal flat metrics. Later, Brendle \cite{Brendle05} proved convergence of the flow for all conformal classes and arbitrary initial metrics, and extended the results to higher dimensions \cite{Brendle07}.

In this paper, we generalize the Yamabe flow to smooth metric measure spaces. 

To explain the results of this article requires some terminology. A \textit{smooth metric measure space} is a four-tuple $(M^n, g, e^{-\phi} dV_g, m)$ of a Riemannian manifold $(M^n,g)$, a smooth measure $e^{-\phi}dV_g$ determined by a function $\phi\in C^{\infty}(M)$ and the Riemannian volume element of $g$, and a dimensional parameter $m\in [0,\infty]$. In the case $m=0$, we require $\phi=0$. We freuqently denote a smooth metric measure space by the triple $(M^n, g, v^{m}dV_g)$, where the measure is $v^{m}dV_g$ and the dimensional parameter is encoded as the exponent of $v$. In accordance with this convention, $v$ and $\phi$ will denote throughout this article functions which are related by $v^m=e^{-\phi}$; when $m=\infty$, this is to be interpreted as the formal definition of the symbol $v^{\infty}$. 

Conformal equivalence between smooth metric measure spaces are defined as the following, see \cite{Case15} for more details.
\begin{defn}\label{condef}
Smooth metric measure spaces $(M^n, g, e^{-\phi}dV_g, m)$ \\and $(M^n, \hat{g}, e^{-\hat{\phi}}dV_{\hat{g}}, m)$ are conformally equivalent if there is a smooth function $\sigma\in C^{\infty}(M)$ such that
\begin{equation}\label{1.2}
(M^n, \hat{g}, e^{-\hat{\phi}}dV_{\hat{g}}, m)=(M^n, e^{\frac{2}{m+n-2}\sigma}g, e^{\frac{m+n}{m+n-2}\sigma}e^{-\phi}dV_g, m).
\end{equation}
In the case $m=0$, conformal equivalence is defined in the classical sense. 
\end{defn}
If we denote $e^{\frac{1}{2}\sigma}$ by $w$, (\ref{1.2}) is equivalent to 
\begin{equation}
(M^n, \hat{g}, e^{-\hat{\phi}}dV_{\hat{g}}, m)=(M^n,w^{\frac{4}{m+n-2}}g,w^{\frac{2(m+n)}{m+n-2}}e^{-\phi}dV_{g}, m),
\end{equation}
which is an alternative way to formulate the conformal equivalence of smooth metric measure spaces.

The \textit{weighted scalar curvature} $R^m_{\phi}$ of a smooth metric measure space is 
\begin{equation}\label{defn1}
R^m_{\phi}:=R+2\Delta \phi-\frac{m+1}{m}|\nabla\phi|^2,
\end{equation}
where $R$ and $\Delta$ are the scalar curvature and the Laplacian associated to the metric $g$, respectively.

In this article, we study a Yamabe-type flow on the smooth metric measure space $(M^n, g, e^{-\phi}dV_g, m)$, $m\in (0,\infty)$, called the weighted Yamabe flow. The definiton of the weighted Yamabe flow arises from the following observation.
\par We consider a family of pairs $(g(t), \phi(t))\in \mathcal{M}_{+}\times C^{\infty}(M)$, where $\mathcal{M}_{+}$ is the space of Riemannian metrics on $M^n$. The fact $(g(t), \phi(t))$ varies within the conformal class $[(g_0, \phi_0)]$ implies that the metric $(e^{\phi(t)})^{\frac{2}{m}}g(t)$ will be fixed. Denoting $\frac{\partial \phi(t)}{\partial t}=\psi(t)$ and $\frac{\partial g(t)}{\partial t}=h(t)g(t)$, we conclude that
\begin{displaymath}
\frac{2}{m}\psi(t)+h(t)=0.
\end{displaymath}
Based on this observation, we define the (normalized) weighted Yamabe flow as:
\begin{equation}\label{flow}
\left\{
\begin{array}{ll}
\frac{\partial g}{\partial t} &=(r^m_{\phi}-R^m_{\phi})g, \\
\frac{\partial \phi}{\partial t} &=\frac{m}{2}(R^m_{\phi}-r^m_{\phi}),
\end{array}
\right.
\end{equation}
where $r^m_{\phi}$ is the mean value of $R^m_{\phi}$; i.e.
\begin{displaymath}
r^m_{\phi}=\frac{\int_M R^m_{\phi} e^{-\phi}dV_g}{\int_M e^{-\phi}dV_g}.
\end{displaymath}
By formulas of first variations in \cite[Theorem 1.174]{Besse87}, direct calculation shows that evolution equations for $R_g$, $\Delta \phi$ and $|\nabla \phi|^2$ are
\begin{align*}
    \frac{\partial R_g}{\partial t}&=(n-1)\Delta R^m_{\phi}+(R^m_{\phi}-r^m_{\phi})R_g\\
    \frac{\partial \Delta \phi}{\partial t} &=(R^m_{\phi}-r^m_{\phi})\Delta\phi-\frac{n-2}{2}g(dR^m_{\phi}, d\phi)+\frac{m}{2}\Delta R^m_{\phi}\\
    \frac{\partial g(\nabla \phi, \nabla \phi)}{\partial t} &=(R^m_{\phi}-r^m_{\phi})g(\nabla \phi, \nabla \phi)+mg(dR^m_{\phi}, d\phi).
\end{align*}
Hence, by the definition of $R^m_{\phi}$, 
\begin{equation}\label{eq:R}
\frac{\partial R^m_{\phi}}{\partial t}
=(n+m-1)\Delta_{ \phi(t)}R^m_{\phi}+R^m_{\phi}(R^m_{\phi}-r^m_{\phi}),
\end{equation}
where $\Delta_{g(t), \phi(t)}$ is the weighted Laplacian on $(M^n, g(t), e^{-\phi(t)}dV_{g(t)}, m)$.
\begin{rem}{The weighted Laplacian:} 
Let $(M^n, g, v^{m}dV_g)$ be a smooth metric measure space. The weighted Laplacian $\Delta_{\phi}: C^{\infty}(M)\to C^{\infty}(M)$ is the operator
\begin{displaymath}
\Delta_{\phi}:=\Delta-\nabla \phi.
\end{displaymath}
It is formally self-adjoint with respect to the measure $e^{-\phi}dV_{g}$, see \cite{Case15} for more details. 
\end{rem}
Equation (\ref{eq:R}) is analogous to the evolution of the scalar curvature along the normalized Yamabe flow. Noting that the flow (\ref{flow}) is subcritical in the sense that $\frac{2(n+m)}{n+m-2}<\frac{2n}{n-2}$. As a result, we can establish the sequential compactness in Proposition \ref{prop:com}, which is the main difference between the classical Yamabe flow and (\ref{flow}). Moreover, we adapt an argument of Brendle \cite{Brendle05} to establish long-time existence and convergence of the weighted Yamabe flow. 
\begin{teor}\label{main}
On a smooth metric measure spaces $(M^n, g, e^{-\phi}dV_g, m)$, where $(M^n, g)$ is a closed Riemannian manifold of dimension $n\geqslant3$, for every choice of the initial metric and the measure, the weighted Yamabe flow (\ref{flow}) exists for all time and converges to a metric with constant weighted scalar curvature. 
\end{teor}
This article is organized as follows. As mentioned above, we first deal with the positive case.

In Section 2 we prove that the conformal factor $w(t)$ can not blow up in finite time by controlling $w(t)$ from above and below on the interval $[0,T]$. The long-time existence follows from this.

Convergence of the weighted Yamabe flow (\ref{flow}) will be based on the following crucial proposition.
\\\textbf{Proposition 3.3.} Let $\{ t_i: i\in \N \}$ be a sequence of times such that $t_i\to \infty$ as $i\to \infty$. Then, we can find a real number $0<\gamma<1$ and a constant $C$ such that, after passing to a subsequence, we have
\begin{equation}
    r^m_{\phi}(t_i)-r^m_{\infty}\leqslant C\left(\int_M w(t_i)^{\frac{2(n+m)}{n+m-2}}|R^m_{\phi}(t_i)-r^m_{\infty}|^{\frac{2(n+m)}{n+m+2}}e^{-\phi_0} dV_{g_0}\right)^{\frac{n+m+2}{2(n+m)}(1+\gamma)}
\end{equation}
for all integers $i$ in that subsequence. 

In Section 3 under Proposition \ref{cri}, we can obtain decay rates of $r^m_{\phi}(t)$ and the uniform upper bound of $|R^m_{\phi}(t)-r^m_{\phi}(t)|$ in $L^2$ norm. Together with the interior regularity theorem, $w(t)$ is uniformaly bounded above and below on $[0,\infty)$, such that the weighted Yamabe flow can converge smoothly.
 
In Section 4 we complete the proof of Proposition \ref{cri} using the spectral theorem of self-adjoint operators and asymptotic analysis.

In Section 5 in the same spirit as \cite{Ye94}, we refine the argument in Section 2 to obtain the uniform bound on $w(t)$ and prove the long-time existence and smooth convergence in the negative case. Besides, in the zero case, we obtain the Harnack inequality such that uniform smooth estimates hold.

\section{Longtime existence}
In this section we collect some basic facts for smooth metric measure spaces and prove various properties of the weighted Yamabe flow that will be used throughout this article.
\begin{defn}
On a smooth metric measure space $(M^n, g, e^{-\phi}dV_g, m)$, which is conformal to $(M^n, g_0, e^{-\phi_0}dV_{g_0}, m)$ in the sense of Definition \ref{condef},
\begin{displaymath}
(M^n, g, e^{-\phi}dV_g, m)=(M^n, w^{\frac{4}{n+m-2}}g_0, w^{\frac{2(m+n)}{n+m-2}}e^{-\phi_0}dV_{g_0}, m),
\end{displaymath} 
analogous to the classical Yamabe problem, we define the normalized energy functional $E(w)$ as
\begin{equation}\label{energy}
     E_{(g_0,\phi_0)}(w)=\frac{\int_M \left(\frac{4(n+m-1)}{n+m-2}L^m_{\phi_0}w,w\right)e^{-\phi_0}dV_{g_0} }{\left( \int_M w^{\frac{2(n+m)}{n+m-2}}e^{-\phi_0}dV_{g_0}\right)^{\frac{n+m-2}{n+m}}},
\end{equation}
where $L^m_{\phi_0}$ is the weighted conformal Laplacian on $(M^n, g_0, e^{-\phi_0}dV_{g_0}, m)$
\begin{displaymath}
L^m_{\phi_0}=-\Delta_{\phi_0}+\frac{n+m-2}{4(n+m-1)}R^m_{\phi_0}.
\end{displaymath}
\end{defn}
\begin{rem}{The weighted scalar curvature:}
More generally, the weighted scalar curvature is defined as
\begin{equation}\label{defn2}
R^m_{\phi}:=R+2\Delta \phi-\frac{m+1}{m}|\nabla\phi|^2+m(m-1)\mu e^{\frac{2\phi}{m}}.
\end{equation}
The smooth metric measure space $(M^n, g, e^{-\phi}dV_g, m)$ can be regarded as the base of the warped product 
\begin{equation}\label{warped}
    \left(M^n\times F^m(\mu), g\oplus e^{-\frac{2\phi}{m}}h \right)
\end{equation}
where $(F^m(\mu),h)$ is the $m$-dimensional simply connected spaceform with constant sectional curvature $\mu$ \cite{Case19}. 
$R^m_{\phi}$ defined in (\ref{defn2}) is the natural analogue of the scalar curvature because it is the scalar curvature of the warped product (\ref{warped}) and plays the role of the scalar curvature in various geometric problems (see \cite{Lott07,Perelman02,Case13,Case15} for more details). The constancy of $R^m_{\phi}$ defined in (\ref{defn1}) is equivalent to the fact that the warped product 
$\left(M^n\times F^m(0), g\oplus e^{-\frac{2\phi}{m}}h \right)$ has constant scalar curvature. Throughout this article, we will adopt (\ref{defn1}) as the definition of the weighted scalar curvature.


\end{rem}

\begin{rem}{The normalized total weighted scalar curvature:}
Under the setting of Definition \ref{energy}, by the transformation law of the weighted scalar curvature in \cite{Case12}, 
\begin{equation}\label{eq:conf}
R^m_{\phi}=\frac{4(n+m-1)}{n+m-2}w^{-\frac{m+n+2}{m+n-2}}L^m_{\phi_0}w,
\end{equation}
the normalized energy $E_{(g_0,\phi_0)}(w)$ is exactly the normalized total weighted scalar curvature of $(M^n, g, e^{-\phi}dV_g, m)$; i.e.
\begin{displaymath}
E_{(g_0,\phi_0)}(w)=E_{(g,\phi)}(1)= \frac{\int_M R^m_{\phi}e^{-\phi}dV_{g} }{\left( {\rm Vol}\left(M^n, e^{-\phi}dV_g\right)\right)^{\frac{n+m-2}{n+m}}}.
\end{displaymath}
We set 
\begin{equation}
    Y_{n,m}[(g, \phi)]=\inf\{E_{(g_0,\phi_0)}(w)| w\in C^{\infty}(M; \R_{+})\}.
\end{equation}
By the transformation law in (\ref{eq:conf}), $Y_{n,m}[(g, \phi)]$ is conformal invariant.
\end{rem}
\begin{rem}
The limit smooth metric measure space $(M,g_{\infty},e^{-{\phi_{\infty}}}dV_{g_{\infty}}, m)$ of (\ref{flow}) has constant weighted scalar curvature; i.e.
\begin{equation}\label{eq:pre}
L^m_{\phi_0}w_{\infty}=\frac{n+m-2}{4(n+m-1)}\lambda w_{\infty}^{\frac{m+n+2}{m+n-2}}
\end{equation}
for some constant $\lambda$, where $w_{\infty}$ is determined by
\begin{displaymath}
\begin{split}
g_{\infty}&=w_{\infty}^{\frac{4}{n+m-2}}g_0, \\
e^{-\phi_{\infty}}&=w_{\infty}^{\frac{2m}{n+m-2}}e^{-\phi_0}.
\end{split} 
\end{displaymath}
Since the energy functional $E$ is subcritical in the sense that $\frac{2(n+m)}{n+m-2}<\frac{2n}{n-2}$, the equation (\ref{eq:pre}) can be solved directly by minimizing the energy. Therefore, we are interested in the weighted Yamabe flow itself, instead of solving (\ref{eq:pre}). 
\end{rem}

As usual in this article, all integrals are computed with respect to the weighted measure $e^{-\phi}dV_g$. We denote $W^{1,2}(M, e^{-\phi}dV_g)$ and $L^{p}(M, e^{-\phi}dV_g)$ respectively the closure of $C^{\infty}(M)$ with respect to the norm
\begin{displaymath}
\begin{split}
    \|w\|_{W^{1,2}(M, e^{-\phi}dV_g)}&:=\left(\int_M \left(|\nabla w|^2+w^2\right)e^{-\phi}dV_g\right)^{\frac{1}{2}}\\
    \|w\|_{L^{p}(M, e^{-\phi}dV_g)}&:=\left(\int_M |w|^p e^{-\phi}dV_g\right)^{\frac{1}{p}}.
\end{split}
\end{displaymath}

As observed in \cite{Case15}, the sign of $Y_{n,m}[(g, \phi)]$ is the same as the sign of the weighted conformal Laplacian.
\begin{prop}[{\cite[Propositon 3.5]{Case15}}]\label{threecase}
Let $(M^n, g, e^{-\phi}dV_g, m)$ be a compact smooth metric measure space and denote
\begin{displaymath}
\lambda_1(L^m_{\phi}):=\inf \left\{\frac{(L^m_{\phi}w,w)}{\|w\|^2_2}| 0\neq w\in W^{1,2}(M, e^{-\phi}dV_g)\right\}.
\end{displaymath}
Then exactly one of the three following statements is true:
\begin{itemize}
    \item $\lambda_1(L^m_{\phi})$ and $Y_{n,m}[(g, \phi)]$ are both positive.
    \item $\lambda_1(L^m_{\phi})$ and $Y_{n,m}[(g, \phi)]$ are both zero.
    \item $\lambda_1(L^m_{\phi})$ and $Y_{n,m}[(g, \phi)]$ are both negative.
\end{itemize}
\end{prop}

In light of the discussion in \cite{Ye94}, in case $Y_{n,m}[(g, \phi)]\leqslant 0$, it is not difficult to show convergence of the weighted Yamabe flow (\ref{flow})  as $t\to \infty$. We leave the proof to section 5 and deal with the positive case firstly. 

In the following, without further comment, we choose $(M^n, g_0, e^{-\phi_0}dV_{g_0}, m)$ to be the initial metric measure space with $Y_{n,m}[(g_0, \phi_0)]> 0$. Since the weighted Yamabe flow preserves the conformal structure, we may write \begin{equation}
\begin{cases}
&g(t)=w(t)^{\frac{4}{n+m-2}}g_0, \\
&e^{-\phi(t)}=w(t)^{\frac{2m}{n+m-2}}e^{-\phi_0},
\end{cases}
\end{equation} 
as the solution of (\ref{flow}) with $(g(0), \phi(0))=(g_0, \phi_0)$. Hence, the weighted Yamabe flow reduces to the following evolution equation for the conformal factor:
\begin{equation}\label{eq:factor}
    \frac{\partial }{\partial t}w(t)^{\frac{n+m+2}{n+m-2}}=\frac{n+m+2}{4}\left(\frac{4(n+m-1)}{n+m-2}\Delta_{\phi_0}w-R^m_{\phi_0}w+r^m_{\phi}w^{\frac{n+m+2}{n+m-2}}\right).
\end{equation}

Since 
\begin{equation}
    \frac{d}{d t}\int_M e^{-\phi(t)}dV_{g(t)}
 =\frac{n+m}{2}\int_M (r^m_{\phi}-R^m_{\phi})e^{-\phi}dV_{g}=0, 
\end{equation}
we may assume that 
\begin{equation}\label{eq:vol}
\int_M e^{-\phi(t)}dV_{g(t)}=1
\end{equation}
for all $t\geqslant 0$. With this normalization, the mean value of the weighted scalar curvature can be written as
\begin{displaymath}
r^m_{\phi}(t)=\int_M R^m_{\phi}(t) e^{-\phi(t)}dV_{g(t)}.
\end{displaymath}
Using the evolution equation (\ref{eq:R}), we obtain
\begin{equation}\label{eq:r}
\frac{d}{d t}r^m_{\phi}(t)
=-\frac{n+m-2}{2}\int_M (r^m_{\phi}-R^m_{\phi})^2e^{-\phi}dV_{g}\leqslant 0.
\end{equation}
Observe that $r^m_{\phi}(t)>0$ since $Y_{n,m}[(g, \phi)]>0$. Hence, $r^m_{\phi}(t)$ can be bounded above and below i.e.
\begin{equation}\label{decreasing}
    0<r^m_{\phi}(t)\leqslant r^m_{\phi}(0).
\end{equation}
In particular, the function $t\mapsto r^m_{\phi}(t)$ is decreasing.
\begin{prop}\label{unilow}
The weighted scalar curvature of the pair $(g(t), \phi(t))$ satisfies
\begin{equation}
    \inf_{M} R^m_{\phi}(t)\geqslant \min \left\{ \inf_{M}R^m_{\phi}(0), 0 \right\}
\end{equation}
for all $t\geqslant 0$.
\end{prop}
\begin{proof}
According to the evolution equation (\ref{eq:R}),
\begin{displaymath}
\begin{split}
    \frac{\partial R^m_{\phi}}{\partial t} & =(n+m-1)\Delta_{ \phi(t)}R^m_{\phi}+R^m_{\phi}(R^m_{\phi}-r^m_{\phi})\\
    & \geqslant (n+m-1)\Delta_{ \phi(t)}R^m_{\phi}-r^m_{\phi}R^m_{\phi}.
\end{split}
\end{displaymath}
Since $r^m_{\phi}(t)>0$, the assertion follows
from the maximum principle. 
\end{proof}
\begin{rem}{Positivity along the flow: }
In particular, if $R^m_{\phi}(0)>0$, we have $R^m_{\phi}(t)>0$ for all $t\geqslant 0$.
\end{rem}
\begin{rem}{Lower bound of $\phi$: }
From above discussion, we obtain that
\begin{equation}
    \frac{\partial \phi}{\partial t}\geqslant\frac{m}{2}(\inf_{M}R^m_{\phi}(0)-r^m_{\phi}(0)),
\end{equation}
which provides a lower bound of $\phi$.
\end{rem}
For abbreviation, let
\begin{displaymath}
\sigma=\max \left\{ \sup_{M}(1-R^m_{\phi}(0)), 1\right\},
\end{displaymath}
so that $R^m_{\phi}(t)+\sigma\geqslant 1$ for all $t\geqslant 0$.

The following result is similar to that in \cite{Brendle05}. Our arguments mostly follow those of Brendle, but in our setting, all integrals are computed using the weighted measure.
\begin{lem}\label{eq: ev}
For every $p>2$, we have
\begin{displaymath}
\begin{split}
    &\frac{d}{d t} \int_M (R^m_{\phi(t)}+\sigma)^{p-1}e^{-\phi(t)}dV_{g(t)}\\
   &=-\frac{4(n+m-1)(p-2)}{p-1}\int_M |d(R^m_{\phi(t)}+\sigma)^{\frac{p-1}{2}}|^2_{g(t)}e^{-\phi(t)}dV_{g(t)}\\
    & -\frac{n+m+2-2p}{2}\int_M((R^m_{\phi(t)}+\sigma)^{p-1}-(r^m_{\phi(t)}+\sigma)^{p-1})(R^m_{\phi(t)}-r^m_{\phi(t)})e^{-\phi(t)}dV_{g(t)}\\
    & -(p-1)\int_M \sigma ((R^m_{\phi(t)}+\sigma)^{p-2}-(r^m_{\phi(t)}+\sigma)^{p-2})(R^m_{\phi(t)}-r^m_{\phi(t)})e^{-\phi(t)}dV_{g(t)}.
\end{split}
\end{displaymath}
\end{lem}
\begin{proof}
This follows immediately from the evolution equation (\ref{eq:R}) and integration by parts. Notice that $r^m_{\phi}(t)$ and $\sigma$ are independent of points in $M$.
\begin{displaymath}
\begin{split}
    &\frac{d}{d t} \int_M (R^m_{\phi(t)}+\sigma)^{p-1}e^{-\phi(t)}dV_{g(t)}\\
    &= \int_M (p-1)(R^m_{\phi(t)}+\sigma)^{p-2}\frac{\partial  R^m_{\phi}}{\partial t} e^{-\phi(t)}dV_{g(t)}\\
    & +\int_M  (R^m_{\phi(t)}+\sigma)^{p-1}\frac{n+m}{2} (r^m_{\phi}-R^m_{\phi})e^{-\phi(t)}dV_{g(t)}\\
    &=-\frac{4(n+m-1)(p-2)}{p-1}\int_M |d(R^m_{\phi(t)}+\sigma)^{\frac{p-1}{2}}|^2_{g(t)}e^{-\phi(t)}dV_{g(t)}\\
    & -(\frac{n+m+2-2p}{2})\int_M((R^m_{\phi(t)}+\sigma)^{p-1}-(r^m_{\phi(t)}+\sigma)^{p-1})(R^m_{\phi(t)}-r^m_{\phi(t)})e^{-\phi(t)}dV_{g(t)}\\
    & -(p-1)\int_M \sigma ((R^m_{\phi(t)}+\sigma)^{p-2}-(r^m_{\phi(t)}+\sigma)^{p-2})(R^m_{\phi(t)}-r^m_{\phi(t)})e^{-\phi(t)}dV_{g(t)}.
\end{split}
\end{displaymath}
\end{proof}
\begin{lem}\label{eq:ev2}
For every $p>\max\{ \frac{n+m}{2},2\}$, we have 
\begin{equation}
    \begin{split}
        \frac{d}{d t}\int_M |R^m_{\phi(t)}-r^m_{\phi(t)}|^p e^{-\phi(t)}dV_{g(t)}\leqslant &C \left(\int_M |R^m_{\phi(t)}-r^m_{\phi(t)}|^p e^{-\phi(t)}dV_{g(t)} \right)^{\frac{2p-(n+m)+2}{2p-(n+m)}}\\
        & +C \int_M |R^m_{\phi(t)}-r^m_{\phi(t)}|^p e^{-\phi(t)}dV_{g(t)}
    \end{split}
\end{equation}
for some uniform constant C independent of $t$.
\end{lem}
\begin{proof}
Using (\ref{eq:R}) and (\ref{eq:r}), we obtain
\begin{displaymath}
\begin{split}
    \frac{d}{d t}&\int_M |R^m_{\phi(t)}-r^m_{\phi(t)}|^p e^{-\phi(t)}dV_{g(t)}=p(n+m-1) \cdot\\
    & \int_M |R^m_{\phi(t)}-r^m_{\phi(t)}|^{p-2}(R^m_{\phi(t)}-r^m_{\phi(t)})\left(\Delta_{ \phi(t)}R^m_{\phi}+R^m_{\phi}(R^m_{\phi}-r^m_{\phi})\right) e^{-\phi(t)}dV_{g(t)}\\
    & -\frac{n+m}{2}\int_M (R^m_{\phi(t)}-r^m_{\phi(t)})|R^m_{\phi(t)}-r^m_{\phi(t)}|^p e^{-\phi(t)}dV_{g(t)}\\
    & +\frac{(n+m-2)p}{2} \int_M (R^m_{\phi(t)}-r^m_{\phi(t)})|R^m_{\phi(t)}-r^m_{\phi(t)}|^{p-2} e^{-\phi(t)}dV_{g(t)}\\
    &\times \int_M |R^m_{\phi(t)}-r^m_{\phi(t)}|^2 e^{-\phi(t)}dV_{g(t)}.
\end{split}
\end{displaymath}
Moreover, we have
\begin{displaymath}
    \begin{split}
    &\frac{d}{d t}\int_M |R^m_{\phi(t)}-r^m_{\phi(t)}|^p e^{-\phi(t)}dV_{g(t)}\\
    &=-\frac{4(p-1)(n+m-1)}{p}\int_M \left(L^m_{\phi}|R^m_{\phi(t)}-r^m_{\phi(t)}|^{\frac{p}{2}},|R^m_{\phi(t)}-r^m_{\phi(t)}|^{\frac{p}{2}}\right)e^{-\phi(t)}dV_{g(t)}\\
    & +\left(\frac{(n+m-2)(p-1)}{p}+p-\frac{n+m}{2}\right)\int_M (R^m_{\phi(t)}-r^m_{\phi(t)})|R^m_{\phi(t)}-r^m_{\phi(t)}|^p e^{-\phi(t)}dV_{g(t)}\\
    & +\left(\frac{(n+m-2)(p-1)}{p}+p\right) \int_M r^m_{\phi(t)}|R^m_{\phi(t)}-r^m_{\phi(t)}|^p e^{-\phi(t)}dV_{g(t)} \\
    & +\frac{(n+m-2)p}{2} \int_M (R^m_{\phi(t)}-r^m_{\phi(t)})|R^m_{\phi(t)}-r^m_{\phi(t)}|^{p-2} e^{-\phi(t)}dV_{g(t)}\\
    &\times \int_M |R^m_{\phi(t)}-r^m_{\phi(t)}|^2 e^{-\phi(t)}dV_{g(t)}.
\end{split}
\end{displaymath}

Since $Y_{n,m}[(g_0, \phi_0)]>0$ and the function $t\mapsto r^m_{\phi}(t)$ is decreasing, we obtain 
\begin{displaymath}
\begin{split}
   \frac{d}{d t}&\int_M |R^m_{\phi(t)}-r^m_{\phi(t)}|^p e^{-\phi(t)}dV_{g(t)}\\
   \leqslant&-\frac{(n+m-2)(p-1)}{p}Y_{n,m}\left(\int_M |R^m_{\phi(t)}-r^m_{\phi(t)}|^{\frac{p(n+m)}{n+m-2}}e^{-\phi(t)}dV_{g(t)}\right)^{\frac{n+m-2}{n+m}}\\
    & +\left(\frac{(n+m-2)(p-1)}{p}+p-\frac{n+m}{2}\right)\int_M |R^m_{\phi(t)}-r^m_{\phi(t)}|^{p+1} e^{-\phi(t)}dV_{g(t)} \\
    & +\left(\frac{(n+m-2)(p-1)}{p}+p\right) \int_M r^m_{\phi_0}|R^m_{\phi(t)}-r^m_{\phi(t)}|^p e^{-\phi(t)}dV_{g(t)} \\
    & +\frac{(n+m-2)p}{2} \int_M (R^m_{\phi(t)}-r^m_{\phi(t)})|R^m_{\phi(t)}-r^m_{\phi(t)}|^{p-2} e^{-\phi(t)}dV_{g(t)}\\
    &\times \int_M |R^m_{\phi(t)}-r^m_{\phi(t)}|^2 e^{-\phi(t)}dV_{g(t)}.
\end{split}
\end{displaymath}
By H\"older's inequality in $L^{p}(M, e^{-\phi(t)}dV_{g(t)})$ and (\ref{eq:vol}), we have
\begin{displaymath}
\begin{split}
    \int_M (R^m_{\phi(t)}-r^m_{\phi(t)})&|R^m_{\phi(t)}-r^m_{\phi(t)}|^{p-2} e^{-\phi(t)}dV_{g(t)}\times \int_M |R^m_{\phi(t)}-r^m_{\phi(t)}|^2 e^{-\phi(t)}dV_{g(t)}\\
    \leqslant &\left(\int_M |R^m_{\phi(t)}-r^m_{\phi(t)}|^{p} e^{-\phi(t)}dV_{g(t)}\right)^{\frac{p+1}{p}},
\end{split}
\end{displaymath}
and
\begin{displaymath}
\begin{split}
    \int_M |R^m_{\phi(t)}-r^m_{\phi(t)}|^{p+1} e^{-\phi(t)}dV_{g(t)}& \leqslant \left(\int_M |R^m_{\phi(t)}-r^m_{\phi(t)}|^{p} e^{-\phi(t)}dV_{g(t)}\right)^{\frac{2p-(n+m)+2}{2p}}\\
    &  \times \left(\int_M |R^m_{\phi(t)}-r^m_{\phi(t)}|^{\frac{p(n+m)}{n+m-2}} e^{-\phi(t)}dV_{g(t)}\right)^{\frac{n+m-2}{2p}}.
\end{split}
\end{displaymath}
Moreover, since $p>\max\{ \frac{n+m}{2},2\}$, we denote $\frac{2p-(n+m)+2}{2p-(n+m)}$ by $\bar{p}$. By Young's inequality, we obtain
\begin{displaymath}
\begin{split}
    \int_M |R^m_{\phi(t)}-r^m_{\phi(t)}|^{p+1} & e^{-\phi(t)}dV_{g(t)}
 \leqslant  C_1\left(\int_M |R^m_{\phi(t)}-r^m_{\phi(t)}|^{p} e^{-\phi(t)}dV_{g(t)}\right)^{\bar{p}}\\
    & +C_2\left(\int_M |R^m_{\phi(t)}-r^m_{\phi(t)}|^{\frac{p(n+m)}{n+m-2}} e^{-\phi(t)}dV_{g(t)}\right)^{\frac{n+m-2}{n+m}}.
\end{split}
\end{displaymath}
From this, the assertion follows.
\end{proof}
In order to bound the solution $w(t)$ above and below in the interval $[0, T]$, we need the following two lemmas.
\begin{lem}\label{lem:Harnack}
Let $P$ be a smooth function on $(M^n, g, e^{-\phi}dV_g, m)$. Moreover, assume that $u$ is a positive function on $M$ such that 
\begin{displaymath}
-\frac{4(n+m-1)}{n+m-2}\Delta_{\phi} u+Pu\geqslant 0.
\end{displaymath}
Then, there exists a constant $C$, depending only on $g$, $\phi$ and $P$, such that
\begin{equation}
    \int_M u e^{-\phi}dV_g\leqslant C \inf_{M} u.
\end{equation}
Moreover, we have
\begin{equation}
    \int_{M} u^{\frac{2(n+m)}{n+m-2}}e^{-\phi}dV_g\leqslant C\inf_M u(\sup_M u)^{\frac{n+m+2}{n+m-2}}.
\end{equation}
\end{lem}
\begin{proof}
Fix $r>0$ sufficiently small. Notice that the weighted Laplacian $\Delta_{\phi}$ has the same second-order terms as the classical Laplacian. The difference only occurs on lower order terms. Therefore, the weak Harnack inequality for linear elliptic equations \cite[Theorem 8.18]{DN01} can still hold in the weighted case, i.e. we obtain
\begin{displaymath}
\int_{B_{2r}(x)} u e^{-\phi}dV_g\leqslant e^{-\inf\phi}\int_{B_{2r}(x)} u dV_g \leqslant e^{-\inf\phi}L_0\inf_{B_r(x)} u
\end{displaymath}
for some constant $L_0$.\\
The assertion follows from the same argument as that in \cite[Proposition A.2]{Brendle05}.
\end{proof}
\begin{prop}\label{bound}
Given any $T>0$, we can find positive constants $C(T)$ and $c(T)$ such that
\begin{displaymath}
\sup_M w(t)\leqslant C(T)
\end{displaymath}
and
\begin{displaymath}
\inf_M w(t)\geqslant c(T)
\end{displaymath}
for all $0\leqslant t\leqslant T$.
\end{prop}
\begin{proof}
The function $w(t)$ satisfies
\begin{displaymath}
\frac{\partial}{\partial t}w(t)=-\frac{n+m-2}{4}(R^m_{\phi(t)}-r^m_{\phi(t)})w(t)\leqslant \frac{n+m-2}{4}(r^m_{\phi_0}+\sigma)w(t).
\end{displaymath}
Hence, 
\begin{displaymath}
\frac{\partial}{\partial t}\ln(w(t))\leqslant \frac{n+m-2}{4}(r^m_{\phi}(0)+\sigma).
\end{displaymath}
We conclude that $\sup_M w(t)\leqslant C(T)$ for all $0\leqslant t\leqslant T$. Hence, if we define
\begin{displaymath}
P=R^m_{\phi_0}+\sigma(\sup_{0\leqslant t\leqslant T}\sup_M w(t))^{\frac{4}{n+m-2}},
\end{displaymath}
then we obtain
\begin{equation}
    \begin{split}
    & -\frac{4(n+m-1)}{n+m-2}\Delta_{ \phi_0} w(t)+Pw(t)\\
    & \geqslant -\frac{4(n+m-1)}{n+m-2}\Delta_{ \phi_0} w(t)+R^m_{\phi_0}w(t)+\sigma w(t)^{\frac{n+m+2}{n+m-2}}\\
    & =(R^m_{\phi(t)}+\sigma)w(t)^{\frac{n+m+2}{n+m-2}}\geqslant 0
    \end{split}
\end{equation}
for all $0\leqslant t\leqslant T$. By Lemma \ref{lem:Harnack} and (\ref{eq:vol}), we can find a positive constant $c(T)$ such that 
\begin{displaymath}
\inf_M w(t)(\sup_M w(t))^{\frac{n+m+2}{n+m-2}}\geqslant c(T)
\end{displaymath}
for all $0\leqslant t\leqslant T$. Since $\sup_M w(t)\leqslant C(T)$, the assertion follows.
\end{proof}
\begin{prop}
Let $0<\alpha< \min\{\frac{2m}{n+m},1\}$. Given any $T>0$, there exists a constant $C(T)$ such that
\begin{equation}
    |w(x_1, t_1)-w(x_2, t_2)|\leqslant C(T)((t_1-t_2)^{\frac{\alpha}{2}}+d(x_1, x_2)^{\alpha})
\end{equation}
for all $x_1, x_2\in M$ and all $t_1, t_2\in [0,T]$ satisfying $0<t_1-t_2<1$.
\end{prop}
\begin{proof}
Using Lemma \ref{eq: ev} with $p=\frac{n+m+2}{2}$, we obtain for all $0\leqslant t\leqslant T$
\begin{displaymath}
\frac{d}{dt}\int_{M}(R^m_{\phi(t)}+\sigma)^{\frac{n+m}{2}}e^{-\phi(t)}dV_{g(t)}\leqslant 0,
\end{displaymath}
which implies for all $0\leqslant t\leqslant T$
\begin{displaymath}
\int_{M}(R^m_{\phi(t)}+\sigma)^{\frac{n+m}{2}}e^{-\phi(t)}dV_{g(t)}\leqslant C.
\end{displaymath}
Hence, 
\begin{equation}\label{eq: p}
\begin{split}
& \left(\int_M |R^m_{\phi(t)}-r^m_{\phi(t)}|^{\frac{n+m}{2}} e^{-\phi(t)}dV_{g(t)}\right)^{\frac{2}{n+m}} \\
& \leqslant \left(\int_{M}(R^m_{\phi(t)}+\sigma)^{\frac{n+m}{2}}e^{-\phi(t)}dV_{g(t)}\right)^{\frac{2}{n+m}}+(r^m_{\phi(t)}+\sigma)\\
& \leqslant C.
\end{split}
\end{equation}
Let $\alpha=2-\frac{n}{p}$, where $\frac{n}{2}<p<\frac{n+m}{2}$, $m>0$. Using (\ref{eq:factor}), (\ref{eq: p}) and Proposition \ref{bound}, we obtain
\begin{equation}
    \int_M \left|-\frac{4(n+m-1)}{n+m-2}\Delta_{\phi_0} w(t)+R^m_{\phi_0}w(t)\right|^p e^{-\phi_0}dV_{g_0}\leqslant C(T)
\end{equation}
and
\begin{equation}
    \int_M \left|\frac{\partial}{\partial t}w(t)\right|^p e^{-\phi(t)}dV_{g(t)}\leqslant C(T)
\end{equation}
for all $t\in [0,T]$. By the embedding $W^{2,p}(M)\hookrightarrow C^{0,\alpha}(M)$, the first inequality implies that 
\begin{displaymath}
|w(x_1, t)-w(x_2, t)|\leqslant C(T)d(x_1, x_2)^{\alpha}
\end{displaymath}
for all $x_1, x_2\in M$ and all $t\in [0,T]$. Using the second inequality, we obtain
\begin{displaymath}
\begin{split}
    & |w(x, t_1)-w(x, t_2)|\\
    & \leqslant C(t_1-t_2)^{-\frac{n}{2}}\int_{B_{\sqrt{t_1-t_2}}(x)}|w(x, t_1)-w(x, t_2)|e^{-\phi_0}dV_{g_0}\\
    & \leqslant C(t_1-t_2)^{-\frac{n}{2}}\int_{B_{\sqrt{t_1-t_2}}(x)}|w(t_1)-w(t_2)|e^{-\phi_0}dV_{g_0}+C(T)(t_1-t_2)^{\frac{\alpha}{2}}\\
    & \leqslant C(t_1-t_2)^{-\frac{n-2}{2}}\sup_{t_2\leqslant t\leqslant t_1}\int_{B_{\sqrt{t_1-t_2}}(x)}\left|\frac{\partial}{\partial t}w(t)\right|e^{-\phi_0}dV_{g_0}+C(T)(t_1-t_2)^{\frac{\alpha}{2}}\\
    & \leqslant C(t_1-t_2)^{\frac{\alpha}{2}}\sup_{t_2\leqslant t\leqslant t_1}\left(\int_{B_{\sqrt{t_1-t_2}}(x)}\left|\frac{\partial}{\partial t}w(t)\right|^p e^{-\phi_0}dV_{g_0}\right)^{\frac{1}{p}}+C(T)(t_1-t_2)^{\frac{\alpha}{2}}\\
    & \leqslant C(T) (t_1-t_2)^{\frac{\alpha}{2}},
\end{split}
\end{displaymath}
for all $x\in M$ and all $t_1, t_2\in [0,T]$ satisfying $0<t_1-t_2<1$. This proves the assertion.
\end{proof}
We can now use the standard regularity theory for parabolic equations \cite[Section 3, Theorem 5]{Friedman13} to show that all higher order derivatives of $w$ are uniformly bounded on every fixed time interval $[0, T]$. Therefore, the flow exists for all time.

\section{Proof of the main result assuming Proposition \ref{cri}}
In this section, we will prove Theorem \ref{main} based on Proposition \ref{cri}. In the following, $c$ and $C$ are positive constants whose value are independent of $t$ and may change from line to line.
\begin{prop}
Fix $\max\{\frac{n+m}{2}, 2\}<p<\frac{n+m+2}{2}$. Then, we have
\begin{displaymath}
\lim_{t\to\infty}\int_M |R^m_{\phi(t)}-r^m_{\phi(t)}|^p e^{-\phi(t)}dV_{g(t)}=0.
\end{displaymath}
\end{prop}
\begin{proof}
It follows from Lemma \ref{eq: ev} that
\begin{displaymath}
\begin{split}
    \frac{d}{d t} &\int_M (R^m_{\phi(t)}+\sigma)^{p-1}e^{-\phi(t)}dV_{g(t)}\leqslant -\left(\frac{n+m+2}{2}-p\right)\cdot \\ &\int_M((R^m_{\phi(t)}+\sigma)^{p-1}-(r^m_{\phi(t)}+\sigma)^{p-1})(R^m_{\phi(t)}-r^m_{\phi(t)})e^{-\phi(t)}dV_{g(t)}.
\end{split}
\end{displaymath}
Since $p>2$, we have
\begin{displaymath}
((R^m_{\phi(t)}+\sigma)^{p-1}-(r^m_{\phi(t)}+\sigma)^{p-1})(R^m_{\phi(t)}-r^m_{\phi(t)})\geqslant c|R^m_{\phi(t)}-r^m_{\phi(t)}|^p
\end{displaymath}
for a suitable constant $c>0$. Since $p<\frac{n+m+2}{2}$, it follows that
\begin{displaymath}
\frac{d}{d t} \int_M (R^m_{\phi(t)}+\sigma)^{p-1}e^{-\phi(t)}dV_{g(t)}\leqslant -c\int_M|R^m_{\phi(t)}-r^m_{\phi(t)}|^pe^{-\phi(t)}dV_{g(t)}.
\end{displaymath}
Integrating with respect to $t$ yields
\begin{displaymath}
\int_0^{\infty}\int_M |R^m_{\phi(t)}-r^m_{\phi(t)}|^pe^{-\phi(t)}dV_{g(t)}\leqslant C,
\end{displaymath}
hence,
\begin{displaymath}
\liminf_{t\to \infty} \int_M |R^m_{\phi(t)}-r^m_{\phi(t)}|^p e^{-\phi(t)}dV_{g(t)}=0.
\end{displaymath}
On the other hand, since $p>\max\{\frac{n+m}{2}, 2\}$, by Lemma \ref{eq:ev2}, we have
\begin{equation}
    \begin{split}
        \frac{d}{d t}\int_M |R^m_{\phi(t)}-r^m_{\phi(t)}|^p e^{-\phi(t)}dV_{g(t)}\leqslant &C \left(\int_M |R^m_{\phi(t)}-r^m_{\phi(t)}|^p e^{-\phi(t)}dV_{g(t)} \right)^{\frac{2p-(n+m)+2}{2p-(n+m)}}\\
        & +C \int_M |R^m_{\phi(t)}-r^m_{\phi(t)}|^p e^{-\phi(t)}dV_{g(t)}
    \end{split}
\end{equation}
From this, the assertion follows.
\end{proof}
Hence, if we define
\begin{equation}
    r^m_{\infty}=\lim_{t\to\infty}r^m_{\phi(t)},
\end{equation}
then we obtain the following result:
\begin{cor}\label{cor:p}
For every $1<p<\frac{n+m+2}{2}$, we have
\begin{equation}
\lim_{t\to\infty}\int_M |R^m_{\phi(t)}-r^m_{\infty}|^p e^{-\phi(t)}dV_{g(t)}=0.
\end{equation}
\end{cor}
The proof of the main result will be based on the following proposition. The proof of this critical proposition will occupy Section 4.
\begin{prop}\label{cri}
Let $\{ t_i: i\in \N \}$ be a sequence of times such that $t_i\to \infty$ as $i\to \infty$. Then, we can find a real number $0<\gamma<1$ and a constant $C$ such that, after passing to a subsequence, we have
\begin{equation}
    r^m_{\phi(t_i)}-r^m_{\infty}\leqslant C\left(\int_M w(t_i)^{\frac{2(n+m)}{n+m-2}}|R^m_{\phi(t_i)}-r^m_{\infty}|^{\frac{2(n+m)}{n+m+2}}e^{-\phi_0} dV_{g_0}\right)^{\frac{n+m+2}{2(n+m)}(1+\gamma)}
\end{equation}
for all integers $i$ in that subsequence. 
\end{prop}
Note that $\gamma$ and $C$ may depend on the sequence $\{ t_i: i\in \N \}$. The following result is an immediate consequence of Proposition \ref{cri}.
\begin{prop}\label{prop:1}
There exists real numbers $0<\gamma<1$ and $t_0>0$ such that
\begin{equation}
r^m_{\phi(t)}-r^m_{\infty}\leqslant \left(\int_M w(t)^{\frac{2(n+m)}{n+m-2}}|R^m_{\phi(t)}-r^m_{\infty}|^{\frac{2(n+m)}{n+m+2}}e^{-\phi_0}dV_{g_0}\right)^{\frac{n+m+2}{2(n+m)}(1+\gamma)}
\end{equation}
for all $t\geqslant t_0$.
\end{prop}
\begin{proof}
Suppose this is not true. Then, there exists a sequence of times $\{ t_i: i\in \N \}$ such that $t_i\geqslant i$ and 
\begin{displaymath}
r^m_{\phi(t_i)}-r^m_{\infty}\geqslant \left(\int_M w(t_i)^{\frac{2(n+m)}{n+m-2}}|R^m_{\phi(t_i)}-r^m_{\infty}|^{\frac{2(n+m)}{n+m+2}}e^{-\phi_0}dV_{g_0}\right)^{\frac{n+m+2}{2(n+m)}(1+\frac{1}{i})}.
\end{displaymath}
We now apply Proposition \ref{cri} to this sequence $\{ t_i: i\in \N \}$. Hence, there exists an infinite subset $I\subset \N$, real numbers $0<\gamma<1$ and $C$ such that
\begin{displaymath}
r^m_{\phi(t_i)}-r^m_{\infty}\leqslant C\left(\int_M w(t_i)^{\frac{2(n+m)}{n+m-2}}|R^m_{\phi(t_i)}-r^m_{\infty}|^{\frac{2(n+m)}{n+m+2}}e^{-\phi_0} dV_{g_0}\right)^{\frac{n+m+2}{2(n+m)}(1+\gamma)}
\end{displaymath}
for all $i\in I$. Thus, we conclude that
\begin{displaymath}
1\leqslant C\left(\int_M w(t_i)^{\frac{2(n+m)}{n+m-2}}|R^m_{\phi(t_i)}-r^m_{\infty}|^{\frac{2(n+m)}{n+m+2}}e^{-\phi_0} dV_{g_0}\right)^{\frac{n+m+2}{2(n+m)}(\gamma-\frac{1}{i})}
\end{displaymath}
for all $i\in I$. 

On the other hand, using Corollary \ref{cor:p} with $p=\frac{2(n+m)}{n+m+2}<\frac{n+m+2}{2}$ and $w(t_i)^{\frac{2(n+m)}{n+m-2}}e^{-\phi_0} dV_{g_0}=e^{-\phi(t_i)} dV_{g(t_i)}$, we have
\begin{displaymath}
\lim_{i\to \infty} \int_M |R^m_{\phi(t_i)}-r^m_{\infty}|^{\frac{2(n+m)}{n+m+2}}e^{-\phi(t_i)} dV_{g(t_i)}=0.
\end{displaymath}
Therefore, if $i$ is sufficiently large, 
\begin{displaymath}
\begin{split}
    &\lim_{i\to \infty} \left(\int_M |R^m_{\phi(t_i)}-r^m_{\infty}|^{\frac{2(n+m)}{n+m+2}}e^{-\phi(t_i)} dV_{g(t_i)}\right)^{\frac{n+m+2}{2(n+m)}(\gamma-\frac{1}{i})}\\
    & \leqslant \lim_{i\to \infty} \left(\int_M |R^m_{\phi(t_i)}-r^m_{\infty}|^{\frac{2(n+m)}{n+m+2}}e^{-\phi(t_i)} dV_{g(t_i)}\right)^{\frac{n+m+2}{2(n+m)}\frac{\gamma}{2}}=0.
\end{split}
\end{displaymath}
 This is a contradiction.
\end{proof}
\begin{prop}\label{uni}
We have
\begin{equation}
    \int_0^{\infty}\left(\int_M w(t)^{\frac{2(n+m)}{n+m-2}}|R^m_{\phi(t)}-r^m_{\phi(t)}|^2 e^{-\phi_0} dV_{g_0}\right)^{\frac{1}{2}} dt \leqslant C.
\end{equation}
\end{prop}
\begin{proof}
It follows from Proposition \ref{prop:1} that
\begin{displaymath}
\begin{split}
r^m_{\phi(t)}-r^m_{\infty}& \leqslant \left(\int_M w(t)^{\frac{2(n+m)}{n+m-2}}|R^m_{\phi(t)}-r^m_{\phi(t)}|^{\frac{2(n+m)}{n+m+2}}e^{-\phi_0} dV_{g_0}\right)^{\frac{n+m+2}{2(n+m)}(1+\gamma)} \\
& +C(r^m_{\phi(t)}-r^m_{\infty})^{1+\gamma},
\end{split}
\end{displaymath}
hence,
\begin{equation}
    r^m_{\phi(t)}-r^m_{\infty} \leqslant \left(\int_M w(t)^{\frac{2(n+m)}{n+m-2}}|R^m_{\phi(t)}-r^m_{\phi(t)}|^{\frac{2(n+m)}{n+m+2}}e^{-\phi_0} dV_{g_0}\right)^{\frac{n+m+2}{2(n+m)}(1+\gamma)}
\end{equation}
if $t$ is sufficiently large. Therefore, by H\"older's inequality and (\ref{eq:vol}), we have
\begin{equation}\label{decay}
    \begin{split}
    \frac{d}{d t}(r^m_{\phi(t)}&-r^m_{\infty}) =-\frac{n+m-2}{2}\int_M (R^m_{\phi(t)}-r^m_{\phi(t)})^2 w(t)^{\frac{2(n+m)}{n+m-2}}e^{-\phi_0}dV_{g_0}\\
    &\leqslant -\frac{n+m-2}{2} \left(\int_M w(t)^{\frac{2(n+m)}{n+m-2}}|R^m_{\phi(t)}-r^m_{\phi(t)}|^{\frac{2(n+m)}{n+m+2}}e^{-\phi_0} dV_{g_0}\right)^{\frac{n+m+2}{n+m}}\\
    & \leqslant -\frac{n+m-2}{2}(r^m_{\phi(t)}-r^m_{\infty})^{\frac{2}{1+\gamma}}.
    \end{split}
\end{equation}
This implies 
\begin{equation}
    \frac{d}{d t}(r^m_{\phi(t)}-r^m_{\infty})^{-\frac{1-\gamma}{1+\gamma}}\geqslant c.
\end{equation}
From this, it follows that if $t$ is sufficiently large,
\begin{equation}
    r^m_{\phi(t)}-r^m_{\infty}\leqslant Ct^{-\frac{1-\gamma}{1+\gamma}}.
\end{equation}
Moreover, integrating the first equality in (\ref{decay}) from $T$ to $2T$ yields
\begin{equation}
    r^m_{\phi(T)}-r^m_{\phi(2T)}=\frac{n+m-2}{2}\int_T^{2T}\int_M (R^m_{\phi(t)}-r^m_{\phi(t)})^2 w(t)^{\frac{2(n+m)}{n+m-2}}e^{-\phi_0}dV_{g_0} dt.
\end{equation}
Using H\"older's inequality, we obtain
\begin{equation}
    \begin{split}
        \int_T^{2T}&\left(\int_M (R^m_{\phi(t)}-r^m_{\phi(t)})^2 w(t)^{\frac{2(n+m)}{n+m-2}}e^{-\phi_0}dV_{g_0} \right)^{\frac{1}{2}}dt\\
        & \leqslant \left(T \int_T^{2T}\int_M (R^m_{\phi(t)}-r^m_{\phi(t)})^2 w(t)^{\frac{2(n+m)}{n+m-2}}e^{-\phi_0}dV_{g_0} dt \right)^{\frac{1}{2}}\\
        & \leqslant \left(\frac{2}{n+m-2}T(r^m_{\phi(T)}-r^m_{\phi(2T)})\right)^{\frac{1}{2}}\\
        & \leqslant CT^{-\frac{\gamma}{1-\gamma}}
    \end{split}
\end{equation}
if $T$ is sufficiently large. Since $0<\gamma<1$, we conclude that
\begin{equation}
    \begin{split}
        \int_0^{\infty}&(\int_M (R^m_{\phi(t)}-r^m_{\phi(t)})^2 w(t)^{\frac{2(n+m)}{n+m-2}}e^{-\phi_0}dV_{g_0} )^{\frac{1}{2}}dt\\
        & = \int_0^{1}(\int_M (R^m_{\phi(t)}-r^m_{\phi(t)})^2 w(t)^{\frac{2(n+m)}{n+m-2}}e^{-\phi_0}dV_{g_0} )^{\frac{1}{2}}dt\\
        & +\sum^{\infty}_{k=0}\int_{2^k}^{2^{k+1}}(\int_M (R^m_{\phi(t)}-r^m_{\phi(t)})^2 w(t)^{\frac{2(n+m)}{n+m-2}}e^{-\phi_0}dV_{g_0} )^{\frac{1}{2}}dt\\
        & \leqslant C\left(1+\sum_{k=0}^{\infty}2^{-\frac{\gamma}{1-\gamma}k}\right)\leqslant C.
    \end{split}
\end{equation}
This proves the assertion.
\end{proof}
\begin{prop}\label{prop2}
Given any $\eta_0>0$, we can find a real number $r>0$ such that
\begin{equation}
    \int_{B_r(x)}w(t)^{\frac{2(n+m)}{n+m-2}}e^{-\phi_0}dV_{g_0}\leqslant \eta_0
\end{equation}
for all $x\in M$ and $t\geqslant 0$.
\end{prop}
\begin{proof}
We can find a real number $T>0$ such that
\begin{equation}
    \int_T^{\infty}\left(\int_M (R^m_{\phi(t)}-r^m_{\phi(t)})^2 w(t)^{\frac{2(n+m)}{n+m-2}}e^{-\phi_0}dV_{g_0} \right)^{\frac{1}{2}}dt\leqslant \frac{\eta_0}{n}.
\end{equation}
We now choose a real number $r>0$ such that
\begin{equation}
    \int_{B_r(x)}w(t)^{\frac{2(n+m)}{n+m-2}}e^{-\phi_0}dV_{g_0}\leqslant \frac{\eta_0}{2}
\end{equation}
for all $x\in M$ and $0\leqslant t\leqslant T$. Using the evolution equation (\ref{eq:vol}) and H\"older's inequality, we have
\begin{equation}\label{eq:vol2}
\begin{split}
    \frac{d}{dt}\int_{B_r(x)}e^{-\phi(t)}dV_{g(t)}&=\frac{n+m}{2}\int_{B_r(x)} (r^m_{\phi(t)}-R^m_{\phi(t)})e^{-\phi(t)}dV_{g(t)}\\
    &\leqslant \frac{n+m}{2}\left(\int_M (R^m_{\phi(t)}-r^m_{\phi(t)})^2 e^{-\phi(t)}dV_{g(t)} \right)^{\frac{1}{2}}.
\end{split}
\end{equation}
Integrating (\ref{eq:vol2}) from $T$ to $t$ yields
\begin{equation}
    \begin{split}
       \int_{B_r(x)}&w(t)^{\frac{2(n+m)}{n+m-2}}e^{-\phi_0}dV_{g_0}\leqslant \int_{B_r(x)}w(T)^{\frac{2(n+m)}{n+m-2}}e^{-\phi_0}dV_{g_0}\\
       & +\frac{n+m}{2}\int_T^{\infty}\left(\int_M (R^m_{\phi(t)}-r^m_{\phi(t)})^2 w(t)^{\frac{2(n+m)}{n+m-2}}e^{-\phi_0}dV_{g_0} \right)^{\frac{1}{2}}dt\leqslant \eta_0
    \end{split}
\end{equation}
for all $x\in M$ and $t\geqslant T$. This proves the assertion.
\end{proof}
\begin{lem}\label{lem1}
Let $p=\frac{2(n+m)}{n+m-2}$ and $q>\frac{n}{2}$. There are positive constants $\eta_1$, $C$ such that if 
\begin{equation}
\begin{cases}
&g=w^{\frac{4}{n+m-2}}g_0, \\
&e^{-\phi}=w^{\frac{2m}{n+m-2}}e^{-\phi_0},
\end{cases}
\end{equation}
is conformal to $(M^n, g_0, e^{-\phi_0}dV_{g_0}, m)$ and
\begin{align*}
    &\int_{B_r(x)}e^{-\phi}dV_{g}\leqslant 1,\\
    &\int_{B_r(x)}|R^m_{\phi}|^qe^{-\phi}dV_{g}\leqslant \eta_1,
\end{align*}
where $B_r(x)$ is the ball with respect to $g_0$ with $r<1$, then
\begin{equation}
    w(x)\leqslant Cr^{-\frac{n}{p}}\left( \int_{B_r(x)}e^{-\phi}dV_{g}\right)^{\frac{1}{p}}.
\end{equation}
\end{lem}
\begin{proof}
By the smoothness of the conformal factor $w(t)$, there exists $r_0$ a real number such that $r_0<r$ and
\begin{displaymath}
(r-s)^{\frac{n}{p}}\sup_{B_s(x)} w\leqslant (r-r_0)^{\frac{n}{p}}\sup_{B_{r_0}(x)} w
\end{displaymath}
for all $s<r$. Moreover, we choose a point $x_0\in \overline{B_{r_0}(x)}$ such that
\begin{displaymath}
\sup_{B_{r_0}(x)} w=w(x_0).
\end{displaymath}
Notice that in (\ref{eq:conf}), the conformal weighted Laplacian $L^m_{\phi_0}$ has the same second-order terms as the classical Laplacian $\Delta_{g_0}$. The difference only occurs on lower order terms. Using a standard interior estimate for linear elliptic equations in \cite[Theorem 8.17]{DN01}, we obtain
\begin{equation}
\begin{split}
s^{\frac{n}{p}}w(x_0)&\leqslant C\left(\int_{B_s(x_0)}w^{p}e^{-\phi_0}dV_{g_0}\right)^{\frac{1}{p}}\\
& +Cs^{\frac{n}{p}+2-\frac{n}{q}}\left(\int_{B_s(x_0)}\left|\frac{4(n+m-1)}{(n+m-2)}L^m_{\phi_0}w\right|^q e^{-\phi_0}dV_{g_0}\right)^{\frac{1}{q}}
\end{split}
\end{equation}
for $s\leqslant \frac{r-r_0}{2}$. From this, it follows that
\begin{equation}
\begin{split}
    s^{\frac{n}{p}}w(x_0)&\leqslant C\left(\int_{B_s(x_0)}e^{-\phi}dV_{g}\right)^{\frac{1}{p}}\\
    & +Cs^{\frac{n}{p}+2-\frac{n}{q}}\left(\int_{B_s(x_0)}w^{(p-1)q-p}|R^m_{\phi}|^q e^{-\phi}dV_{g} \right)^{\frac{1}{q}}
    \end{split}
\end{equation}
for $s\leqslant \frac{r-r_0}{2}$. By definition of $r_0$ and $x_0$, we have
\begin{equation}
    \sup_{B_{\frac{r-r_0}{2}}(x_0)} w\leqslant \sup_{B_{\frac{r+r_0}{2}}(x)} w \leqslant 2^{\frac{n}{p}} \sup_{B_{r_0}(x)} w=2^{\frac{n}{p}} w(x_0).
\end{equation}
Notice that $2-n+\frac{2n}{p}>0$ for $p=\frac{2(n+m)}{n+m-2}<\frac{2n}{n-2}$, and $s<r<1$. Hence, we can find a fixed constant $K$ such that
\begin{equation}\label{ineq:3.22}
    \begin{split}
    s^{\frac{n}{p}}w(x_0)&\leqslant K\left(\int_{B_s(x_0)}e^{-\phi}dV_{g}\right)^{\frac{1}{p}}\\
    & +K(s^{\frac{n}{p}}w(x_0))^{(p-1)-\frac{p}{q}}\left(\int_{B_s(x_0)}|R^m_{\phi}|^q e^{-\phi}dV_{g} \right)^{\frac{1}{q}}
    \end{split}
\end{equation}
for $s\leqslant \frac{r-r_0}{2}$. We now choose $\eta_1>0$ such that
\begin{displaymath}
(2K)^{(p-1)-\frac{p}{q}}\eta_1^{\frac{1}{q}}\leqslant \frac{1}{2}.
\end{displaymath}

We claim that $(\frac{r-r_0}{2})^{\frac{n}{p}}w(x_0)\leqslant 2K$. Indeed, if $(\frac{r-r_0}{2})^{\frac{n}{p}}w(x_0)\geqslant 2K$, then we may apply inequality (\ref{ineq:3.22}) with $s=(\frac{2K}{w(x_0)})^{\frac{p}{n}}\leqslant \frac{r-r_0}{2}$. This yields
\begin{displaymath}
\begin{split}
2K\leqslant & K\left(\int_{B_r(x_0)}e^{-\phi}dV_{g}\right)^{\frac{1}{p}}\\
& +K(2K)^{(p-1)-\frac{p}{q}} \left(\int_{B_r(x_0)}|R^m_{\phi}|^q e^{-\phi}dV_{g} \right)^{\frac{1}{q}},
\end{split}
\end{displaymath}
hence
\begin{displaymath}
2K\leqslant K+(2K)^{(p-1)-\frac{p}{q}}\eta_1^{\frac{1}{q}}.
\end{displaymath}

Using (\ref{ineq:3.22}) with $s=\frac{r-r_0}{2}$, we obtain
\begin{equation}
\begin{split}
(\frac{r-r_0}{2})^{\frac{n}{p}}w(x_0) & \leqslant  K\left(\int_{B_r(x_0)}e^{-\phi}dV_{g}\right)^{\frac{1}{p}}\\
& +K(2K)^{(p-2)-\frac{p}{q}} \left(\int_{B_r(x_0)}|R^m_{\phi}|^q e^{-\phi}dV_{g} \right)^{\frac{1}{q}}\cdot (\frac{r-r_0}{2})^{\frac{n}{p}}w(x_0).
\end{split}
\end{equation}
This implies
\begin{equation}
    \begin{split}
(\frac{r-r_0}{2})^{\frac{n}{p}}w(x_0) \leqslant & K\left(\int_{B_r(x_0)}e^{-\phi}dV_{g}\right)^{\frac{1}{p}}\\
& +\frac{1}{2}(2K)^{(p-1)-\frac{p}{q}} \eta_1^{\frac{1}{q}}\cdot (\frac{r-r_0}{2})^{\frac{n}{p}}w(x_0),
\end{split}
\end{equation}
hence
\begin{displaymath}
(\frac{r-r_0}{2})^{\frac{n}{p}}w(x_0) \leqslant 2 K\left(\int_{B_r(x_0)}e^{-\phi}dV_{g}\right)^{\frac{1}{p}}.
\end{displaymath}
Thus, we conclude that
\begin{displaymath}
r^{\frac{n}{p}}w(x)\leqslant (r-r_0)^{\frac{n}{p}}w(x_0)\leqslant 2^{\frac{n}{p}+1}K\left(\int_{B_r(x_0)}e^{-\phi}dV_{g}\right)^{\frac{1}{p}}.
\end{displaymath}
This proves the assertion.
\end{proof}
\begin{prop}\label{pro:bound}
The function $w(t)$ satisfies
\begin{equation}
    \sup_M w(t)\leqslant C
\end{equation}
and 
\begin{equation}
    \inf_M w(t)\geqslant c
\end{equation}
for all $t\geqslant 0$. Here, $C$ and $c$ are positive constants independent of $t$.
\end{prop}
\begin{proof}
Fix $\frac{n}{2}<q<p<\frac{n+m+2}{2}$. By Corollary \ref{cor:p}, we have
\begin{displaymath}
\int_{M}|R^m_{\phi(t)}|^pe^{-\phi(t)}dV_{g(t)}\leqslant C,
\end{displaymath}
for some constant $C$ independent of $t$. By Proposition \ref{prop2}, we can find a constant $r>0$ independent of $t$ such that
\begin{displaymath}
\int_{B_r(x)}e^{-\phi(t)}dV_{g(t)}\leqslant \eta_0
\end{displaymath}
for all $x\in M$ and $t\geqslant 0$. Using H\"{o}lder's inequality, we obtain
\begin{displaymath}
\begin{split}
\int_{B_r(x)}|R^m_{\phi(t)}|^qe^{-\phi(t)}dV_{g(t)}\leqslant \left({\rm Vol}(B_r(x))\right)^{\frac{p-q}{p}}\left(\int_{B_r(x)}|R^m_{\phi(t)}|^pe^{-\phi(t)}dV_{g(t)}\right)^{\frac{q}{p}}.
\end{split}
\end{displaymath}
Hence, if we choose $\eta_0$ sufficiently small, then we have
\begin{displaymath}
\int_{B_r(x)}|R^m_{\phi(t)}|^qe^{-\phi(t)}dV_{g(t)}\leqslant \eta_1
\end{displaymath}
for all $x\in M$ and $t\geqslant 0$. Here, $\eta_1$ is the constant appearing in Lemma \ref{lem1}. Applying Lemma \ref{lem1} at the maximum point of $w(t)$ over $M$, 
\begin{equation}
    \sup_{M}w(t)\leqslant Cr^{-\frac{n}{p}}\left( \int_{B_r(x)}e^{-\phi(t)}dV_{g(t)}\right)^{\frac{1}{p}}.
\end{equation}
By (\ref{eq:vol}), we conclude that $w(t)$ is uniformly bounded above. Hence, if we define
\begin{displaymath}
P=R^m_{\phi_0}+\sigma(\sup_{t\geqslant 0}\sup_M w(t))^{\frac{4}{n+m-2}},
\end{displaymath}
then we obtain
\begin{equation}
    \begin{split}
    & -\frac{4(n+m-1)}{n+m-2}\Delta_{ \phi_0} w(t)+Pw(t)\\
    & \geqslant -\frac{4(n+m-1)}{n+m-2}\Delta_{\phi_0} w(t)+R^m_{\phi_0}w(t)+\sigma w(t)^{\frac{n+m+2}{n+m-2}}\\
    & =(R^m_{\phi(t)}+\sigma)w(t)^{\frac{n+m+2}{n+m-2}}\geqslant 0.
    \end{split}
\end{equation}
By Lemma \ref{lem:Harnack} and (\ref{eq:vol}), we can find a positive constant $c$ independent of $t$ such that 
\begin{displaymath}
\inf_M w(t)(\sup_M w(t))^{\frac{n+m+2}{n+m-2}}\geqslant c
\end{displaymath}
for all $t\geqslant 0$. Since $\sup_M w(t)\leqslant C$, the assertion follows.
\end{proof}
\begin{prop}\label{Holder}
Let $0<\alpha< \min\{\frac{2m}{n+m},1\}$. Then, the function $w(t)$ satisfies 
\begin{equation}
    |w(x_1, t_1)-w(x_2, t_2)|\leqslant C((t_1-t_2)^{\frac{\alpha}{2}}+d(x_1, x_2)^{\alpha})
\end{equation}
for all $x_1, x_2\in M$ and $0<t_1-t_2<1$. Here, $C$ is a positive constant independent of $t_1$ and $t_2$.
\end{prop}
\begin{proof}
Let $\alpha=2-\frac{n}{p}$, where $\frac{n}{2}<p<\frac{n+m}{2}$, $m>0$. Using (\ref{cor:p}) and Proposition \ref{pro:bound}, we obtain
\begin{equation}
    \int_M \left|-\frac{4(n+m-1)}{n+m-2}\Delta_{\phi_0} w(t)+R^m_{\phi_0}w(t)\right|^p e^{-\phi_0}dV_{g_0}\leqslant C
\end{equation}
and
\begin{equation}
    \int_M \left|\frac{\partial}{\partial t}w(t)\right|^p e^{-\phi(t)}dV_{g(t)}\leqslant C
\end{equation}
where $C$ is a positive constant independent of $t$. By the embedding $W^{2,p}(M)\hookrightarrow C^{0.\alpha}(M)$, the first inequality implies that 
\begin{displaymath}
|w(x_1, t)-w(x_2, t)|\leqslant Cd(x_1, x_2)^{\alpha}
\end{displaymath}
for all $x_1, x_2\in M$ and all $t\geqslant 0$. Using the second inequality, we obtain
\begin{displaymath}
\begin{split}
    & |w(x, t_1)-w(x, t_2)|\\
    & \leqslant C(t_1-t_2)^{-\frac{n}{2}}\int_{B_{\sqrt{t_1-t_2}}(x)}|w(x, t_1)-w(x, t_2)|e^{-\phi_0}dV_{g_0}\\
    & \leqslant C(t_1-t_2)^{-\frac{n}{2}}\int_{B_{\sqrt{t_1-t_2}}(x)}|w(t_1)-w(t_2)|e^{-\phi_0}dV_{g_0}+C(t_1-t_2)^{\frac{\alpha}{2}}\\
    & \leqslant C(t_1-t_2)^{-\frac{n-2}{2}}\sup_{t_2\leqslant t\leqslant t_1}\int_{B_{\sqrt{t_1-t_2}}(x)}|\frac{\partial}{\partial t}w(t)|e^{-\phi_0}dV_{g_0}+C(t_1-t_2)^{\frac{\alpha}{2}}\\
    & \leqslant C(t_1-t_2)^{\frac{\alpha}{2}}\sup_{t_2\leqslant t\leqslant t_1}\left(\int_{B_{\sqrt{t_1-t_2}}(x)}|\frac{\partial}{\partial t}w(t)|^p e^{-\phi_0}dV_{g_0}\right)^{\frac{1}{p}}+C(t_1-t_2)^{\frac{\alpha}{2}}\\
    & \leqslant C (t_1-t_2)^{\frac{\alpha}{2}}.
\end{split}
\end{displaymath}
for all $x\in M$ and all $t_1, t_2$ satisfying $0<t_1-t_2<1$. This proves the assertion.
\end{proof}
In light of the foregoing argument at the end of Section 2, we derive uniform estimates for all higher order derivatives of $w(t)$, $t\geqslant 0$. The uniqueness of the asymptotic limit follows from Proposition \ref{uni}. This completes the proof of the main result.

\section{Proof of the Critical Proposition}
Let $\{ t_i: i\in \N \}$ be a sequence of times such that $t_i\to \infty$ as $i\to \infty$. For abbreviation, let $w_i=w(t_i)$.
The normalization condition implies that
\begin{displaymath}
\int_M e^{-\phi_i} dV_{g_i}=1, 
\end{displaymath}
where
\begin{equation}
\begin{cases}
&g_i=w_i^{\frac{4}{n+m-2}}g_0, \\
&e^{-\phi_i}=w_i^{\frac{2m}{n+m-2}}e^{-\phi_0}.
\end{cases}
\end{equation}
Hence
\begin{equation}\label{cond1}
    \int_M w_i^{\frac{2(n+m)}{n+m-2}} e^{-\phi_0} dV_{g_0}=1
\end{equation}
for all $i\in \N$. Moreover, it follows from Corollary \ref{cor:p} that
\begin{displaymath}
\int_M |R^m_{\phi(t_i)}-r^m_{\infty}|^{\frac{2(n+m)}{n+m+2}} e^{-\phi(t_i)}dV_{g(t_i)}\to 0,
\end{displaymath}
hence
\begin{equation}\label{cond2}
    \int_M \left|\frac{4(n+m-1)}{n+m-2}\Delta_{\phi_0} w_i-R^m_{\phi_0}w_i+r^m_{\infty}w_i^{\frac{n+m+2}{n+m-2}}\right|^{\frac{2(n+m)}{n+m+2}} e^{-\phi_0}dV_{g_0}\to 0
\end{equation}
as $i\to \infty$.

Using the standard elliptic theory, we have the following compactness result.
\begin{prop}
Let $\{ w_i: i\in \N \}$ be a sequence of positive functions satisfying (\ref{cond1}) and (\ref{cond2}). After passing to a subsequence if necessary, $\{ w_i: i\in \N \}$ converges to a positive smooth function $w_{\infty}$ satisfying the equation:
\begin{displaymath}
\frac{4(n+m-1)}{n+m-2}\Delta_{ \phi_0} w_{\infty}-R^m_{\phi}(0)w_{\infty}+r^m_{\infty}w_{\infty}^{\frac{n+m+2}{n+m-2}}=0.
\end{displaymath}
\end{prop}
\begin{proof}
Noticing that $\frac{n+m+2}{n+m-2}<\frac{n+2}{n-2}$, the assertion follows from the standard elliptic theory \cite[Section 8, Theorem 3]{Evans10}.
\end{proof}
In order to prove the critical proposition, we need the following results.
\begin{prop}\label{prop:com}
The exists a sequence of smooth functions $\{ \psi_{a}: a\in \N \}$ and a sequence of positive real numbers $\{ \lambda_a: a\in \N \}$ with the following properties:\\
(i) For every $a\in \N$, the function $\psi_a$ satisfies the equation
\begin{equation}
\frac{4(n+m-1)}{n+m-2}\Delta_{\phi_0} \psi_a-R^m_{\phi_0}\psi_a+\lambda_{a}w_{\infty}^{\frac{4}{n+m-2}}\psi_a=0.
\end{equation}
(ii) For all $a,b\in\N$, we have
\begin{equation}
\int_M w_{\infty}^{\frac{4}{n+m-2}}\psi_a\psi_b e^{-\phi_0}dV_{g_0}=\left\{
             \begin{array}{lr}
             0, & a\neq b \\
             1, & a=b
             \end{array}
\right..
\end{equation}
(iii) The span of $\{ \psi_{a}: a\in \N \}$ is dense in $L^2(M, e^{-\phi_0}dV_{g_0})$.\\
(iv) $\lambda_a\to\infty$ as $a\to \infty$.
\end{prop}
\begin{proof}
Consider the linear operator
\begin{displaymath}
\psi\mapsto w_{\infty}^{-\frac{4}{n+m-2}}\left(\frac{4(n+m-1)}{n+m-2}\Delta_{ \phi_0} \psi-R^m_{\phi_0}\psi\right).
\end{displaymath}
This operator is symmetric with respect to the inner product
\begin{displaymath}
(\psi_1,\psi_2)\mapsto \int_M w_{\infty}^{\frac{4}{n+m-2}} \psi_1\psi_2 e^{-\phi_0}dV_{g_0}
\end{displaymath}
on $L^2(M, e^{-\phi_0}dV_{g_0})$. Hence, the assertion follows from the spectral theorem.
\end{proof}
Let $A$ be a maximal finite subset of $\N$ such that $\lambda_a\leqslant\frac{n+m+2}{n+m-2}r^m_{\infty}$ for all $a\in A$. We denote by $\Pi$ the projection operator
\begin{equation}
\begin{split}
    \Pi f & =\sum_{a\notin A}\left(\int_M \psi_a f e^{-\phi_0}dV_{g_0}\right)w_{\infty}^{\frac{4}{n+m-2}} \psi_a\\
    & =f-\sum_{a\in A}\left(\int_M \psi_a f e^{-\phi_0}dV_{g_0}\right)w_{\infty}^{\frac{4}{n+m-2}} \psi_a
    \end{split}
\end{equation}
In the rest of this section, for simplicity, we denote $W^{1,2}(M, e^{-\phi_0}dV_{g_0})$ and $L^{p}(M, e^{-\phi_0}dV_{g_0})$ by $W^{1,2}(M)$ and $L^{p}(M)$, respectively.
\begin{lem}\label{lem2}
For every $1\leqslant p<\infty$, we can find a constant $C$ such that
\begin{equation}
\begin{split}
    \| f\|_{L^p(M)}&\leqslant C\left\| \frac{4(n+m-1)}{n+m-2}\Delta_{\phi_0} f-R^m_{\phi_0}f+\frac{n+m+2}{n+m-2}r^m_{\infty}w_{\infty}^{\frac{4}{n+m-2}}f\right\|_{L^p(M)}\\
    & + C\sup_{a\in A}\left|\int_M w_{\infty}^{\frac{4}{n+m-2}}\psi_a f e^{-\phi_0}dV_{g_0}\right|.
    \end{split}
\end{equation}
\end{lem}
\begin{proof}
Assume that is not true. By compactness, we can find a function $f\in L^p(M) $ satisfying $\| f\|_{L^p(M)}=1$,
\begin{equation}
    \int_M w_{\infty}^{\frac{4}{n+m-2}}\psi_a f e^{-\phi_0}dV_{g_0}=0
\end{equation}
for all $a\in A$ and
\begin{equation}
    \frac{4(n+m-1)}{n+m-2}\Delta_{\phi_0} f-R^m_{\phi_0}f+\frac{n+m+2}{n+m-2}r^m_{\infty}w_{\infty}^{\frac{4}{n+m-2}}f=0
\end{equation}
in the sense of distributions. Hence, if we use the function $\psi_a$ as a test function, then we obtain
\begin{displaymath}
(\lambda_a-\frac{n+m+2}{n+m-2}r^m_{\infty})\int_M w_{\infty}^{\frac{4}{n+m-2}}\psi_a f e^{-\phi_0}dV_{g_0}=0
\end{displaymath}
for all $a\in \N$. In particular, we have
\begin{displaymath}
\int_M w_{\infty}^{\frac{4}{n+m-2}}\psi_a f e^{-\phi_0}dV_{g_0}=0
\end{displaymath}
for all $a\notin A$. Thus, we conclude that $f=0$. This is a contradiction.
\end{proof}
\begin{lem}
~\\
(i) There exists a constant $C$ such that
\begin{equation}
    \begin{split}
    &\| f\|_{L^{\frac{n+m+2}{n+m-2}}(M)}\leqslant C\sup_{a\in A}\left|\int_M w_{\infty}^{\frac{4}{n+m-2}}\psi_a f e^{-\phi_0}dV_{g_0}\right|\\
    & +C\left\| \Pi\left(\frac{4(n+m-1)}{n+m-2}\Delta_{ \phi_0} f-R^m_{\phi_0}f+\frac{n+m+2}{n+m-2}r^m_{\infty}w_{\infty}^{\frac{4}{n+m-2}}f\right)\right\|_{L^{s}(M)},
    \end{split}
\end{equation}
where $s=\frac{n(n+m+2)}{n(n+m-2)+2(n+m+2)}$.\\
~\\
(ii) There exists a constant $C$ such that
\begin{equation}
    \begin{split}
    \| f\|_{L^1(M)}&\leqslant C\left\| \Pi\left(\frac{4(n+m-1)}{n+m-2}\Delta_{\phi_0} f-R^m_{\phi_0}f+\frac{n+m+2}{n+m-2}r^m_{\infty}w_{\infty}^{\frac{4}{n+m-2}}f\right)\right\|_{L^{1}(M)}\\
    & + C\sup_{a\in A}\left|\int_M w_{\infty}^{\frac{4}{n+m-2}}\psi_a f e^{-\phi_0}dV_{g_0}\right|.
    \end{split}
\end{equation}
\end{lem}
\begin{proof}
(i) It follows from the embedding $W^{2,s}(M)\hookrightarrow L^{\frac{n+m+2}{n+m-2}}(M)$ that
\begin{equation}
    \begin{split}
    \| f\|_{L^{\frac{n+m+2}{n+m-2}}(M)}&\leqslant  C\| f\|_{L^{s}(M)}\\
    & +C\left\| \frac{4(n+m-1)}{n+m-2}\Delta_{\phi_0} f-R^m_{\phi_0}f+\frac{n+m+2}{n+m-2}r^m_{\infty}w_{\infty}^{\frac{4}{n+m-2}}f\right\|_{L^{s}(M)}.
    \end{split}
\end{equation}
Using Lemma \ref{lem2}, we obtain
\begin{equation}
    \begin{split}
    \| f\|_{L^{\frac{n+m+2}{n+m-2}}(M)}&\leqslant C\sup_{a\in A}\left|\int_M w_{\infty}^{\frac{4}{n+m-2}}\psi_a f e^{-\phi_0}dV_{g_0}\right|\\
    & +C\left\| \frac{4(n+m-1)}{n+m-2}\Delta_{ \phi_0} f-R^m_{\phi_0}f+\frac{n+m+2}{n+m-2}r^m_{\infty}w_{\infty}^{\frac{4}{n+m-2}}f\right\|_{L^{s}(M)}.
    \end{split}
\end{equation}
By definition of $\Pi$, we have 
\begin{equation}
    \begin{split}
        \frac{4(n+m-1)}{n+m-2}&\Delta_{\phi_0} f-R^m_{\phi_0}f+\frac{n+m+2}{n+m-2}r^m_{\infty}w_{\infty}^{\frac{4}{n+m-2}}f\\
        =&\Pi\left(\frac{4(n+m-1)}{n+m-2}\Delta_{\phi_0} f-R^m_{\phi_0}f+\frac{n+m+2}{n+m-2}r^m_{\infty}w_{\infty}^{\frac{4}{n+m-2}}f\right)\\
        & -\sum_{a\in A}(\lambda_a-\frac{n+m+2}{n+m-2}r^m_{\infty})\left(\int_M w_{\infty}^{\frac{4}{n+m-2}}\psi_a f e^{-\phi_0}dV_{g_0}\right)w_{\infty}^{\frac{4}{n+m-2}}\psi_a.
    \end{split}
\end{equation}
This implies
\begin{equation}
    \begin{split}
        & \left\| \frac{4(n+m-1)}{n+m-2}\Delta_{\phi_0} f-R^m_{\phi_0}f+\frac{n+m+2}{n+m-2}r^m_{\infty}w_{\infty}^{\frac{4}{n+m-2}}f \right\| _{L^q(M)} \\
        \leqslant & \left\| \Pi\left(\frac{4(n+m-1)}{n+m-2}\Delta_{\phi_0} f-R^m_{\phi_0}f+\frac{n+m+2}{n+m-2}r^m_{\infty}w_{\infty}^{\frac{4}{n+m-2}}f\right)\right\|_{L^{q}(M)}\\
        & +C\sup_{a\in A}\left|\int_M w_{\infty}^{\frac{4}{n+m-1}}\psi_a f e^{-\phi_0}dV_{g_0}\right|.
    \end{split}
\end{equation}
Putting these facts together, the assertion follows.\\
~\\Similar to (i), (ii) follows from Lemma \ref{lem2} and the definition of $\Pi$. 
\end{proof}
\begin{lem}\label{analytic}
There exists a positive real number $\xi$ such that for every vector $z\in \R^{A}$ with $|z|\leqslant \xi$, there exists a smooth function $\bar{w}_z$ such that
\begin{equation}
    \int_M w_{\infty}^{\frac{4}{n+m-2}}\psi_a (\bar{w}_z-w_{\infty}) e^{-\phi_0}dV_{g_0}=z_a
\end{equation}
for all $a\in A$ and
\begin{equation}
    \Pi\left(\frac{4(n+m-1)}{n+m-2}\Delta_{\phi_0} \bar{w}_z-R^m_{\phi_0}\bar{w}_z+r^m_{\infty}\bar{w}_z^{\frac{n+m+2}{n+m-2}}\right)=0.
\end{equation}
Furthermore, the map $z\mapsto \bar{w}_z$ is real analytic.
\end{lem}
\begin{proof}
This is a consequence of the implicit function theorem.
\end{proof}
\begin{lem}\label{asy}
There exists a real number $0<\gamma<1$ such that
\begin{equation}
    \begin{split}
        &E(\bar{w}_z)-E(w_{\infty})\\
        & \leqslant C\sup_{a\in A}\left|\int_M \left(\frac{4(n+m-1)}{n+m-2}\Delta_{\phi_0} \bar{w}_z-R^m_{\phi_0}\bar{w}_z+r^m_{\infty}\bar{w}_z^{\frac{n+m+2}{n+m-2}}\right)\psi_a  e^{-\phi_0}dV_{g_0}\right|^{1+\gamma}.
    \end{split}
\end{equation}
if $z$ is sufficiently small.
\end{lem}
\begin{proof}
Note that the function $z\mapsto E(\bar{w}_z)$ is real analytic. According to results of Lojasiewicz \cite[equation (2.4)]{Simon83}, there exists a real number $0<\gamma<1$ such that
\begin{equation}
    |E(\bar{w}_z)-E(w_{\infty})|\leqslant \sup_{a\in A} \left|\frac{\partial}{\partial z_a}E(\bar{w}_z)\right|^{1+\gamma}
\end{equation}
if $z$ is sufficiently small. For convenience, we define the energy functional $\tilde{E}_{(g_0,\phi_0)}(w)$ as
\begin{equation}
\begin{split}
    \tilde{E}_{(g_0,\phi_0)}(w)&=\frac{\int_M \left(\frac{4(n+m-1)}{n+m-2}L^m_{\phi_0}w,w\right)e^{-\phi_0}dV_{g_0} }{ \int_M w^{\frac{2(n+m)}{n+m-2}}e^{-\phi_0}dV_{g_0}}\\
    & =\frac{\int_M R^m_{\phi}e^{-\phi}dV_{g} }{ {\rm Vol}(M^n, e^{-\phi}dV_g)}.
\end{split}
\end{equation}
The partial derivatives of the function $z\mapsto E(\bar{w}_z)$ are given by
\begin{equation}
    \begin{split}
    \frac{\partial}{\partial z_a}E(\bar{w}_z) =&-2\frac{\int_M \left( \frac{4(n+m-1)}{n+m-2}\Delta_{\phi_0} \bar{w}_z-R^m_{\phi_0}\bar{w}_z+r^m_{\infty}\bar{w}_z^{\frac{n+m+2}{n+m-2}}\right)\tilde{\psi}_{a,z}e^{-\phi_0}dV_{g_0} }{\left( \int_M \bar{w}_z^{\frac{2(n+m)}{n+m-2}}e^{-\phi_0}dV_{g_0}\right)^{\frac{n+m-2}{n+m}}}\\
    & -2(\tilde{E}(\bar{w}_z)-r^m_{\infty})\frac{\int_M \bar{w}_z^{\frac{n+m+2}{n+m-2}}\tilde{\psi}_{a,z}e^{-\phi_0}dV_{g_0} }{\left( \int_M \bar{w}_z^{\frac{2(n+m)}{n+m-2}}e^{-\phi_0}dV_{g_0}\right)^{\frac{n+m-2}{n+m}}}
    \end{split}
\end{equation}
where $\tilde{\psi}_{a,z}=\frac{\partial}{\partial z_a}\bar{w}_z$ for $a\in A$. The function $\tilde{\psi}_{a,z}$ satisfies
\begin{equation}
\int_M w_{\infty}^{\frac{4}{n+m-2}}\tilde{\psi}_{a,z}\psi_b e^{-\phi_0}dV_{g_0}=\left\{
             \begin{array}{lr}
             0, & a\neq b \\
             1, & a=b
             \end{array}
\right.
\end{equation}
for all $a\in A$ and
\begin{displaymath}
\Pi\left(\frac{4(n+m-1)}{n+m-2}\Delta_{ \phi_0}\tilde{\psi}_{a,z} -R^m_{\phi_0}\tilde{\psi}_{a,z}+\frac{n+m+2}{n+m-2}r^m_{\infty}w_{\infty}^{\frac{4}{n+m-2}}\tilde{\psi}_{a,z}\right)=0.
\end{displaymath}
Using the identity
\begin{displaymath}
\Pi\left( \frac{4(n+m-1)}{n+m-2}\Delta_{\phi_0} \bar{w}_z-R^m_{\phi_0}\bar{w}_z+r^m_{\infty}\bar{w}_z^{\frac{n+m+2}{n+m-2}}\right)=0,
\end{displaymath}
we obtain
\begin{equation}
    \begin{split}
    \frac{\partial}{\partial z_a}E(\bar{w}_z) =&-2\frac{\int_M \left( \frac{4(n+m-1)}{n+m-2}\Delta_{\phi_0} \bar{w}_z-R^m_{\phi_0}\bar{w}_z+r^m_{\infty}\bar{w}_z^{\frac{n+m+2}{n+m-2}}\right){\psi}_{a}e^{-\phi_0}dV_{g_0} }{\left( \int_M \bar{w}_z^{\frac{2(n+m)}{n+m-2}}e^{-\phi_0}dV_{g_0}\right)^{\frac{n+m-2}{n+m}}}\\
    & +2\sum_{b\in A}\frac{\left(\int_M \bar{w}_z^{\frac{n+m+2}{n+m-2}}\tilde{\psi}_{a,z}e^{-\phi_0}dV_{g_0} \right)\left(\int_M w_{\infty}^{\frac{4}{n+m-2}}\bar{w}_z\psi_b e^{-\phi_0}dV_{g_0} \right)}{\left( \int_M \bar{w}_z^{\frac{2(n+m)}{n+m-2}}e^{-\phi_0}dV_{g_0}\right)^{\frac{n+m-2}{n+m}}}\\
    & \cdot \frac{\left(\int_M  \frac{4(n+m-1)}{n+m-2}\Delta_{ \phi_0} \bar{w}_z-R^m_{\phi_0}\bar{w}_z+r^m_{\infty}\bar{w}_z^{\frac{n+m+2}{n+m-2}}\right) \psi_b e^{-\phi_0}dV_{g_0}}{\int_M \bar{w}_z^{\frac{2(n+m)}{n+m-2}}e^{-\phi_0}dV_{g_0}}
    \end{split}
\end{equation}
for all $a\in A$. Thus, we obtain that
\begin{displaymath}
\begin{split}
\sup_{a\in A} &\left|\frac{\partial}{\partial z_a}E(\bar{w}_z)\right|\\
& \leqslant C \sup_{a\in A} \left|\int_M \left( \frac{4(n+m-1)}{n+m-2}\Delta_{\phi_0} \bar{w}_z-R^m_{\phi_0}\bar{w}_z+r^m_{\infty}\bar{w}_z^{\frac{n+m+2}{n+m-2}}\right){\psi}_{a}e^{-\phi_0}dV_{g_0}\right|.
\end{split}
\end{displaymath}
From this, the assertion follows.
\end{proof}
By Lemma \ref{analytic}, the function 
\begin{displaymath}
z\to \int_M \left(L^m_{\phi_0}(w_i-\bar{w}_{z}),(w_i-\bar{w}_{z})\right)e^{-\phi_0}dV_{g_0}
\end{displaymath}
is analytical and attains the infimum for $|z|\leqslant \xi$.
For every $i\in \N$, we can find $\bar{w}_{z_i}$ such that $|z_i|\leqslant \xi$ and
\begin{equation}
\begin{split}
    & \int_M \left(L^m_{\phi_0}(w_i-\bar{w}_{z_i}),(w_i-\bar{w}_{z_i})\right)e^{-\phi_0}dV_{g_0}\\
    & \leqslant \int_M \left(L^m_{\phi_0}(w_i-\bar{w}_{z}),(w_i-\bar{w}_{z})\right)e^{-\phi_0}dV_{g_0}
     \end{split}
\end{equation}
for all $|z|\leqslant \xi$. 

\begin{prop}\label{pro1}
We have as $i\to \infty$,
    \begin{equation}\label{4.26}
    \| w_i-\bar{w}_{z_i}\|_{W^{1,2}(M)}\to 0 ~~\mbox{and}~~  z_i\to 0.
    \end{equation}
\end{prop}

\begin{proof}
Notice that by assumptions in Lemma \ref{analytic}, we have $\bar{w}_{0}=w_{\infty}$. Combining this with the definition of $\bar{w}_{z_i}$ yields 
\begin{equation}
\begin{split}
     & \int_M \left(L^m_{\phi_0}(w_i-\bar{w}_{z_i}),(w_i-\bar{w}_{z_i})\right)e^{-\phi_0}dV_{g_0}\\
     & \leqslant \int_M \left(L^m_{\phi_0}(w_i-{w}_{\infty}),(w_i-{w}_{\infty})\right)e^{-\phi_0}dV_{g_0}.
     \end{split}
\end{equation}

By the compactness result in Proposition \ref{prop:com}, the expression on the right-hand side tends to $0$ as $i\to \infty$, i.e we have as $i\to \infty$,
\begin{equation}
    \| w_i-\bar{w}_{z_i}\|_{W^{1,2}(M)}\to 0.
    \end{equation}
and 
\begin{equation}\label{4.27}
    \| \bar{w}_{z_i}-w_{\infty}\|_{W^{1,2}(M)}\to 0,
\end{equation}
which implies that $z_i\to 0$ as $i\to \infty$.
\end{proof}

We now decompose the function $w_i$ as
\begin{equation*}
    w_i=\bar{w}_{z_i}+u_i.
\end{equation*}
Note that the function $u_i$ satisfies
\begin{equation}
    \int_M \left(\frac{4(n+m-1)}{n+m-2}L^m_{\phi_0}u_i,u_i\right)e^{-\phi_0}dV_{g_0}=o(1)
\end{equation}
by Proposition \ref{pro1}.

\begin{prop}\label{pro2}
The function $u_i$ satisfies the following two properties.
\begin{enumerate}
    \item For every $a\in A$, we have
    \begin{equation}\label{o(1)}
        \left|\int_M w_{\infty}^{\frac{4}{n+m-2}}\psi_a u_i e^{-\phi_0}dV_{g_0} \right|\leqslant o(1)\int_M |u_i|e^{-\phi_0}dV_{g_0}.
    \end{equation}
    \item If $i$ is sufficiently large, then we have
    \begin{equation}
    \begin{split}
        \frac{n+m+2}{n+m-2}&r^m_{\infty}\int_M w_{\infty}^{\frac{4}{n+m-2}} u^2_i e^{-\phi_0}dV_{g_0}\\
        &\leqslant (1-c)\int_M \left(\frac{4(n+m-1)}{n+m-2}L^m_{\phi_0}u_i,u_i\right)e^{-\phi_0}dV_{g_0}
        \end{split}
    \end{equation}
    for some positive constant independent of $i$.
\end{enumerate}
\end{prop}
\begin{proof}
\begin{enumerate}
    \item As above, let $\tilde{\psi}_{a,z}=\frac{\partial}{\partial z_a}\bar{w}_z$ for $a\in A$. By the definition of $z_i$, we have
    \begin{equation*}
        \int_M \left(\frac{4(n+m-1)}{n+m-2}L^m_{\phi_0}\tilde{\psi}_{a,z},u_i\right)e^{-\phi_0}dV_{g_0}=0.
    \end{equation*}
    This implies that
    \begin{equation*}
    \begin{split}
        \lambda_a &\int_M w_{\infty}^{\frac{4}{n+m-2}}\psi_a u_i e^{-\phi_0}dV_{g_0}\\
        &= -\int_M \left(\frac{4(n+m-1)}{n+m-2}L^m_{\phi_0}\psi_a,u_i\right)e^{-\phi_0}dV_{g_0}\\
        &= \int_M \left(\frac{4(n+m-1)}{n+m-2}L^m_{\phi_0}\left(\tilde{\psi}_{a,z}-\psi_a\right),u_i\right)e^{-\phi_0}dV_{g_0}.
        \end{split}
    \end{equation*}
    Since $\lambda_a>0$, we conclude that for all $a\in A$
    \begin{equation*}
        \left|\int_M w_{\infty}^{\frac{4}{n+m-2}}\psi_a u_i e^{-\phi_0}dV_{g_0} \right|\leqslant o(1)\int_M |u_i|e^{-\phi_0}dV_{g_0}.
    \end{equation*}
    \item Suppose this is not true. Upon rescaling, we obtain a sequence of function $\{\tilde{u}_i:i\in \N\}$ such that
    \begin{equation}\label{normalization}
        \int_M \left(\frac{4(n+m-1)}{n+m-2}L^m_{\phi_0}\tilde{u}_i,\tilde{u}_i\right)e^{-\phi_0}dV_{g_0}=1
    \end{equation}
    and 
    \begin{equation}\label{normalization2}
        \lim_{i\to \infty}\frac{n+m+2}{n+m-2}r^m_{\infty}\int_M w_{\infty}^{\frac{4}{n+m-2}} \tilde{u}^2_i e^{-\phi_0}dV_{g_0}\geqslant 1.
    \end{equation}
    Observe that
    \begin{equation*}
        \int_M |\tilde{u}_i|^{\frac{2(n+m)}{n+m-2}}e^{-\phi_0}dV_{g_0}\leqslant Y_{n,m}^{-\frac{n+m}{n+m-2}}
    \end{equation*}
    by (\ref{normalization}). By (\ref{normalization}) and (\ref{normalization2}), we conclude that
    \begin{equation*}
        \lim_{i\to \infty}\int_M w_{\infty}^{\frac{4}{n+m-2}} \tilde{u}^2_i e^{-\phi_0}dV_{g_0}>0
    \end{equation*}
    and 
    \begin{equation*}
        \begin{split}
            \lim_{i\to \infty}&\int_M \left(\frac{4(n+m-1)}{n+m-2}L^m_{\phi_0}\tilde{u}_i,\tilde{u}_i\right)e^{-\phi_0}dV_{g_0}\\
            & \leqslant \lim_{i\to \infty}\frac{n+m+2}{n+m-2}r^m_{\infty}\int_M w_{\infty}^{\frac{4}{n+m-2}} \tilde{u}^2_i e^{-\phi_0}dV_{g_0}.
        \end{split}
    \end{equation*}
    Let $\tilde{u}$ be the weak limit of the sequence $\{\tilde{u}_i:i\in \N\}$. Then, the function $\tilde{u}$ satisfies
    \begin{equation*}
        \int_M w_{\infty}^{\frac{4}{n+m-2}} \tilde{u}^2 e^{-\phi_0}dV_{g_0}>0
    \end{equation*}
    and 
    \begin{equation*}
            \int_M \left(\frac{4(n+m-1)}{n+m-2}L^m_{\phi_0}\tilde{u},\tilde{u}\right)e^{-\phi_0}dV_{g_0} \leqslant \frac{n+m+2}{n+m-2}r^m_{\infty}\int_M w_{\infty}^{\frac{4}{n+m-2}} \tilde{u}^2 e^{-\phi_0}dV_{g_0}.
    \end{equation*}
    This implies that
    \begin{equation*}
    \begin{split}
        \sum_{a\in \N}&\lambda_a\left(\int_M w_{\infty}^{\frac{4}{n+m-2}}\psi_a \tilde{u} e^{-\phi_0}dV_{g_0}\right)^2\\
        &\leqslant \sum_{a\in \N}\frac{n+m+2}{n+m-2}r^m_{\infty} \left(\int_M w_{\infty}^{\frac{4}{n+m-2}}\psi_a \tilde{u} e^{-\phi_0}dV_{g_0}\right)^2.
        \end{split}
    \end{equation*}
    Using (\ref{o(1)}), we obtain that for all $a\in A$
    \begin{equation*}
        \int_M w_{\infty}^{\frac{4}{n+m-2}}\psi_a \tilde{u} e^{-\phi_0}dV_{g_0}=0.
    \end{equation*}
    Therefore, we conclude that $\tilde{u}=0$ on $M$. This is a contradiction.
\end{enumerate}
\end{proof}

\begin{cor}\label{pro3}
If $i$ is sufficiently large, then we have
    \begin{equation}
    \begin{split}
        \frac{n+m+2}{n+m-2}&r^m_{\infty}\int_M \bar{w}_{z_i}^{\frac{4}{n+m-2}} u^2_i e^{-\phi_0}dV_{g_0}\\
        &\leqslant (1-c)\int_M \left(\frac{4(n+m-1)}{n+m-2}L^m_{\phi_0}u_i,u_i\right)e^{-\phi_0}dV_{g_0}
        \end{split}
    \end{equation}
    for some positive constant independent of $i$.
\end{cor}
\begin{proof}
The assertion follows from Proposition \ref{pro1} and Proposition \ref{pro2}.
\end{proof}

\begin{lem}\label{est1}
The function $u_i$ satisfies
\begin{equation}
    \| u_i\|_{L^{\frac{n+m+2}{n+m-2}}(M)}\leqslant C\left\| w_i^{\frac{n+m+2}{n+m-2}}(R^m_{\phi_i}-r^m_{\infty})\right\|_{L^{\frac{2(n+m)}{n+m+2}}(M)}
\end{equation}
if $i$ is sufficiently large.
\end{lem}
\begin{proof}
Using the identities
\begin{displaymath}
\frac{4(n+m-1)}{n+m-2}\Delta_{\phi_0} w_i-R^m_{\phi_0}w_i+r^m_{\infty}w_i^{\frac{n+m+2}{n+m-2}}=-w_i^{\frac{n+m+2}{n+m-2}}(R^m_{\phi_i}-r^m_{\infty})
\end{displaymath}
and 
\begin{displaymath}
\Pi\left( \frac{4(n+m-1)}{n+m-2}\Delta_{\phi_0} \bar{w}_z-R^m_{\phi_0}\bar{w}_z+r^m_{\infty}\bar{w}_z^{\frac{n+m+2}{n+m-2}}\right)=0,
\end{displaymath}
we obtain
\begin{equation}
\begin{aligned}
\Pi&\left(\frac{4(n+m-1)}{n+m-2}\Delta_{\phi_0}u_i-R^m_{\phi_0}u_i+\frac{n+m+2}{n+m-2}r^m_{\infty}w_{\infty}^{\frac{4}{n+m-2}}u_i\right)\\
=& \Pi\left(-w_i^{\frac{n+m+2}{n+m-2}}(R^m_{\phi_i}-r^m_{\infty})-\frac{n+m+2}{n+m-2}r^m_{\infty}\left(\bar{w}_{z_i}^{\frac{4}{n+m-2}}-w_{\infty}^{\frac{4}{n+m-2}}\right)u_i \notag\right.
\\
\phantom{=\;\;} 
& \left. +r^m_{\infty}\left(w_i^{\frac{n+m+2}{n+m-2}}-\bar{w}_{z_i}^{\frac{n+m+2}{n+m-2}}+ \frac{n+m+2}{n+m-2}\bar{w}_{z_i}^{\frac{4}{n+m-2}}u_i\right)\right)
\end{aligned}
\end{equation}
Using the inequality, 
\begin{displaymath}
\begin{split}
    &\| u_i\|_{L^{\frac{n+m+2}{n+m-2}}(M)}\leqslant  C\sup_{a\in A}\left|\int_M w_{\infty}^{\frac{4}{n+m-2}}\psi_a u_i
    e^{-\phi_0}{\rm dvol}_{g_0}\right|\\
    & +C\left\| \Pi\left(\frac{4(n+m-1)}{n+m-2}\Delta_{ \phi_0}u_i -R^m_{\phi_0}u_i+\frac{n+m+2}{n+m-2}r^m_{\infty}w_{\infty}^{\frac{4}{n+m-2}}u_i\right)\right\|_{L^{s}(M)},
\end{split}
\end{displaymath}
we conclude that
\begin{displaymath}
\begin{split}
    \|u_i\|&_{L^{\frac{n+m+2}{n+m-2}}(M)}\leqslant  C\left\| \bar{w}_{z_i}^{\frac{n+m+2}{n+m-2}}-w_i^{\frac{n+m+2}{n+m-2}}+ \frac{n+m+2}{n+m-2}\bar{w}_{z_i}^{\frac{4}{n+m-2}}u_i \right\|_{L^{s}(M)}\\
    & +C\left\| w_i^{\frac{n+m+2}{n+m-2}}(R^m_{\phi_i}-r^m_{\infty})\right\|_{L^{s}(M)}+C\left\| (\bar{w}_{z_i}^{\frac{4}{n+m-2}}-w_{\infty}^{\frac{4}{n+m-2}})u_i \right\|_{L^{s}(M)}\\
    & +C\sup_{a\in A}\left|\int_M w_{\infty}^{\frac{4}{n+m-2}}\psi_a u_i
    e^{-\phi_0}dV_{g_0}\right|.
\end{split}
\end{displaymath}
By the compactness of $(M, g)$, up to a subsequence, we can assume that 
\begin{equation}\label{pointcong}
    w_i\to w_{\infty} ~~\mbox{ and }~~ \bar{w}_{z_i}\to w_{\infty} \quad ~~\mbox{ a.e. in }~~ M.
\end{equation}
When $n\geqslant 3$ and $m>0$, $s=\frac{n(n+m+2)}{n(n+m-2)+2(n+m+2)}<\frac{n+m+2}{n+m-2}$. Combining (\ref{pointcong}) with Lebesgue's dominated convergence theorem and H\"older's inequality yields
\begin{equation*}
    \left\| (\bar{w}_{z_i}^{\frac{4}{n+m-2}}-w_{\infty}^{\frac{4}{n+m-2}})u_i \right\|_{\tilde{L}^{s}(M)}=o(1)\left\| u_i \right\|_{\tilde{L}^{\frac{n+m+2}{n+m-2}}(M)}.
\end{equation*}

By (\ref{pointcong}), we have the pointwise estimate
\begin{equation*}
\begin{split}
&\left|\bar{w}_{z_i}^{\frac{n+m+2}{n+m-2}}-w_i^{\frac{n+m+2}{n+m-2}}+ \frac{n+m+2}{n+m-2}\bar{w}_{z_i}^{\frac{4}{n+m-2}}u_i \right|
    \\
    &\leqslant C \bar{w}_{z_i}^{\max\{\frac{4}{n+m-2}-1,0\}}|u_i|^{\min\{\frac{n+m+2}{n+m-2},2\}}+C|u_i|^{\frac{n+m+2}{n+m-2}}
    \end{split}
\end{equation*}
if $i$ is sufficiently large. Moreover, by Proposition \ref{pro1}, we know that $z_i \to 0$ as $i\to \infty.$ Combining this with the fact that the map $z\mapsto\overline{w}_z$ is real analytic and $\overline{w}_{0}=w_{\infty}$, we have that $\bar{w}_{z_i}$ is uniformly bounded if $i$ is sufficiently large.

Hence, by H\"older's inequality, we obtain
\begin{equation*}
    \begin{split}
        &\left\| \bar{w}_{z_i}^{\frac{n+m+2}{n+m-2}}-w_i^{\frac{n+m+2}{n+m-2}}+ \frac{n+m+2}{n+m-2}\bar{w}_{z_i}^{\frac{4}{n+m-2}}u_i \right\|_{\tilde{L}^{s}(M)}\\
        &\leqslant C \left\| |u_i|^{\min\{\frac{n+m+2}{n+m-2},2\}}+|u_i|^{\frac{n+m+2}{n+m-2}} \right\|_{\tilde{L}^{s}(M)}\\
        & \leqslant  C \left\| |u_i|^{\min\{\frac{4}{n+m-2},1\}}+|u_i|^{\frac{4}{n+m-2}} \right\|_{\tilde{L}^{\frac{n}{2}}(M)}\left\| u_i \right\|_{\tilde{L}^{\frac{n+m+2}{n+m-2}}(M)}\\
        & \leqslant o(1)\left\| u_i \right\|_{\tilde{L}^{\frac{n+m+2}{n+m-2}}(M)}.
    \end{split}
\end{equation*}

Moreover, according to Proposition \ref{pro2}, we have
\begin{equation*}
    \sup_{a\in A}\left|\int_M w_{\infty}^{\frac{4}{n+m-2}}\psi_a u_i
    e^{-\phi_0}dV_{g_0}\right|=o(1)\left\| u_i \right\|_{\tilde{L}^{1}(M)}.
\end{equation*}
Putting these facts together, the assertion follows.
\end{proof}
\begin{lem}\label{est2}
The difference $u_i$ satisfies
\begin{equation}
    \left\| u_i\right\|_{L^{1}(M)}\leqslant C\left\| w_i^{\frac{n+m+2}{n+m-2}}(R^m_{\phi_i}-r^m_{\infty})\right\|_{L^{\frac{2(n+m)}{n+m+2}}(M)}
\end{equation}
if $i$ is sufficiently large.
\end{lem}
\begin{proof}
The proof is totally similar to that of Lemma \ref{est1}. We omit it.
\end{proof}
\begin{lem}\label{est3}
We have
\begin{equation}
    \begin{split}
        \sup_{a\in A}&\left|\left(\int_M \left( \frac{4(n+m-1)}{n+m-2}\Delta_{ \phi_0} \bar{w}_{z_i}-R^m_{\phi_0}\bar{w}_{z_i}+r^m_{\infty}\bar{w}_{z_i}^{\frac{n+m+2}{n+m-2}}\right) \psi_a e^{-\phi_0}dV_{g_0}\right)\right|\\
        & \leqslant C\left(\int_M w(t_i)^{\frac{2(n+m)}{n+m-2}}|R^m_{\phi_i}-r^m_{\infty}|^{\frac{2(n+m)}{n+m+2}}e^{-\phi_0} dV_{g_0}\right)^{\frac{n+m+2}{2(n+m)}}
    \end{split}
\end{equation}
if $i$ is sufficiently large.
\end{lem}
\begin{proof}
Integration by parts yields
\begin{displaymath}
\begin{split}
\int_M & \left( \frac{4(n+m-1)}{n+m-2}\Delta_{\phi_0} \bar{w}_{z_i}-R^m_{\phi_0}\bar{w}_{z_i}+r^m_{\infty}\bar{w}_{z_i}^{\frac{n+m+2}{n+m-2}}\right) \psi_a e^{-\phi_0}dV_{g_0}\\
=&\int_M \left( \frac{4(n+m-1)}{n+m-2}\Delta_{\phi_0} {w}_i-R^m_{\phi_0}{w}_i+r^m_{\infty}{w}_i^{\frac{n+m+2}{n+m-2}}\right) \psi_a e^{-\phi_0}dV_{g_0}\\
&+\lambda_a \int_M w_{\infty}^{\frac{4}{n+m-2}} (w_i-\bar{w}_{z_i})\psi_a e^{-\phi_0}dV_{g_0}\\
&-r^m_{\infty}\int_M \left(w_i^{\frac{n+m+2}{n+m-2}}-\bar{w}_{z_i}^{\frac{n+m+2}{n+m-2}}\right)\psi_ae^{-\phi_0}dV_{g_0}.
\end{split}
\end{displaymath}
Using the identities
\begin{displaymath}
\frac{4(n+m-1)}{n+m-2}\Delta_{\phi_0} w_i-R^m_{\phi_0}w_i+r^m_{\infty}w_i^{\frac{n+m+2}{n+m-2}}=-w_i^{\frac{n+m+2}{n+m-2}}(R^m_{\phi_i}-r^m_{\infty}),
\end{displaymath}
we obtain
\begin{displaymath}
\begin{split}
\int_M & \left( \frac{4(n+m-1)}{n+m-2}\Delta_{\phi_0} \bar{w}_{z_i}-R^m_{\phi_0}\bar{w}_{z_i}+r^m_{\infty}\bar{w}_{z_i}^{\frac{n+m+2}{n+m-2}}\right) \psi_a e^{-\phi_0}dV_{g_0}\\
=&-\int_M w_i^{\frac{n+m+2}{n+m-2}}(R^m_{\phi_i}-r^m_{\infty}) \psi_a e^{-\phi_0}dV_{g_0}\\
&+\lambda_a \int_M w_{\infty}^{\frac{4}{n+m-2}} (w_i-\bar{w}_{z_i})\psi_a e^{-\phi_0}dV_{g_0}\\
&-r^m_{\infty}\int_M \left(w_i^{\frac{n+m+2}{n+m-2}}-\bar{w}_{z_i}^{\frac{n+m+2}{n+m-2}}\right)\psi_ae^{-\phi_0}dV_{g_0}.
\end{split}
\end{displaymath}
Using the pointwise estimate
\begin{displaymath}
\left|w_i^{\frac{n+m+2}{n+m-2}}-\bar{w}_{z_i}^{\frac{n+m+2}{n+m-2}}\right|\leqslant C \bar{w}_{z_i}^{\frac{4}{n+m-2}}\left|w_i-\bar{w}_{z_i}\right|+C\left|w_i-\bar{w}_{z_i}\right|^{\frac{n+m+2}{n+m-2}},
\end{displaymath}
we conclude that
\begin{displaymath}
\begin{split}
        \sup_{a\in A}&\left|\left(\int_M \left( \frac{4(n+m-1)}{n+m-2}\Delta_{ \phi_0} \bar{w}_{z_i}-R^m_{\phi_0}\bar{w}_{z_i}+r^m_{\infty}\bar{w}_{z_i}^{\frac{n+m+2}{n+m-2}}\right) \psi_a e^{-\phi_0}dV_{g_0}\right)\right|\\
        & \leqslant C\left(\int_M w(t_i)^{\frac{2(n+m)}{n+m-2}}|R^m_{\phi_i}-r^m_{\infty}|^{\frac{2(n+m)}{n+m+2}}e^{-\phi_0} dV_{g_0}\right)^{\frac{n+m+2}{2(n+m)}}\\
        & +C\left\| w_i-\bar{w}_{z_i}\right\|_{L^{1}(M)}+C\left\| w_i-\bar{w}_{z_i}\right\|^{\frac{n+m+2}{n+m-2}}_{L^{\frac{n+m+2}{n+m-2}}(M)}.
    \end{split}
\end{displaymath}
The assertion follows from Lemma \ref{est1} and \ref{est2}.
\end{proof}
Combining Lemma \ref{asy} and Lemma \ref{est3}, we immediately obtain that
\begin{prop}
$E(\bar{w}_{z_i})$ satisfies the estimate
\begin{equation}
\begin{split}
    E(\bar{w}_{z_i})&-E(w_{\infty})\\
    & \leqslant C\left(\int_M w(t_i)^{\frac{2(n+m)}{n+m-2}}|R^m_{\phi_i}-r^m_{\infty}|^{\frac{2(n+m)}{n+m+2}}e^{-\phi_0} dV_{g_0}\right)^{\frac{n+m+2}{2(n+m)}(1+\gamma)}
\end{split}
\end{equation}
if $i$ is sufficiently large.
\end{prop}
Now, we can prove our critical proposition.
\begin{proof}[Proof of Proposition \ref{cri}:]
Using the transformation law (\ref{eq:conf}), we obtain
\begin{equation*}
    \begin{split}
        r^m_{\phi}(t_i)=&\int_M \left(\frac{4(n+m-1)}{n+m-2}L^m_{\phi_0}w_i,w_i\right)e^{-\phi_0}dV_{g_0}\\
        =&\int_M \left(\frac{4(n+m-1)}{n+m-2}L^m_{\phi_0}\bar{w}_{z_i},\bar{w}_{z_i}\right)e^{-\phi_0}dV_{g_0}\\
        &+2\int_M w_i^{\frac{n+m+2}{n+m-2}}R^m_{\phi}(t_i)u_i e^{-\phi_0}dV_{g_0} \\
        &-\int_M \left(\frac{4(n+m-1)}{n+m-2}L^m_{\phi_0}u_i,u_i\right)e^{-\phi_0}dV_{g_0}.
    \end{split}
\end{equation*}
This implies that
\begin{equation*}
    \begin{split}
        r^m_{\phi}(t_i)=&E(\bar{w}_{z_i})\left(\int_M \bar{w}_{z_i}^{\frac{2(n+m)}{n+m-2}} e^{-\phi_0}dV_{g_0}\right)^{\frac{n+m-2}{n+m}}\\
        &+2\int_M w_i^{\frac{n+m+2}{n+m-2}}\left(R^m_{\phi}(t_i)-r^m_{\infty}\right)u_i e^{-\phi_0}dV_{g_0} \\
        &-\int_M \left(\frac{4(n+m-1)}{n+m-2}\left(L^m_{\phi_0}u_i,u_i\right)-\frac{n+m+2}{n+m-2}r^m_{\infty}\bar{w}_{z_i}^{\frac{4}{n+m-2}}u_i^2\right)e^{-\phi_0}dV_{g_0}\\
        &+r^m_{\infty}\int_M \left(-\frac{n+m+2}{n+m-2}\bar{w}_{z_i}^{\frac{4}{n+m-2}}u_i^2+2w_i^{\frac{n+m+2}{n+m-2}}u_i\right)e^{-\phi_0}dV_{g_0}.
    \end{split}
\end{equation*}
In view of the volume normalization, we have
\begin{equation*}
    \int_M \left(\bar{w}_{z_i}+u_i\right)^{\frac{2(n+m)}{n+m-2}} e^{-\phi_0}dV_{g_0}=1.
\end{equation*}
Furthermore, it is not difficult to show that
\begin{equation*}
\begin{split}
    &\left(\int_M \bar{w}_{z_i}^{\frac{2(n+m)}{n+m-2}} e^{-\phi_0}dV_{g_0}\right)^{\frac{n+m-2}{n+m}}-1\\
    &\leqslant \frac{n+m-2}{n+m}\left(\int_M \bar{w}_{z_i}^{\frac{2(n+m)}{n+m-2}} e^{-\phi_0}dV_{g_0}\right)-\frac{n+m-2}{n+m},
    \end{split}
\end{equation*}
hence
\begin{equation*}
    \begin{split}
    &\left(\int_M \bar{w}_{z_i}^{\frac{2(n+m)}{n+m-2}} e^{-\phi_0}dV_{g_0}\right)^{\frac{n+m-2}{n+m}}-1\\
    &\leqslant \int_M\left( \frac{n+m-2}{n+m}\bar{w}_{z_i}^{\frac{2(n+m)}{n+m-2}}-\frac{n+m-2}{n+m}\left(\bar{w}_{z_i}+u_i\right)^{\frac{2(n+m)}{n+m-2}} \right)e^{-\phi_0}dV_{g_0}.
    \end{split}
\end{equation*}
It follows that
    \begin{equation*}
    \begin{split}
        r^m_{\phi}(t_i)\leqslant &r^m_{\infty}+\left(E(\bar{w}_{z_i})-r^m_{\infty}\right)\left(\int_M \bar{w}_{z_i}^{\frac{2(n+m)}{n+m-2}} e^{-\phi_0}dV_{g_0}\right)^{\frac{n+m-2}{n+m}}\\
        &+2\int_M w_i^{\frac{n+m+2}{n+m-2}}\left(R^m_{\phi}(t_i)-r^m_{\infty}\right)u_i e^{-\phi_0}dV_{g_0} \\
        &-\int_M \left(\frac{4(n+m-1)}{n+m-2}\left(L^m_{\phi_0}u_i,u_i\right)-\frac{n+m+2}{n+m-2}r^m_{\infty}\bar{w}_{z_i}^{\frac{4}{n+m-2}}u_i^2\right)e^{-\phi_0}dV_{g_0}\\
        &+r^m_{\infty}\int_M \left(-\frac{n+m+2}{n+m-2}\bar{w}_{z_i}^{\frac{4}{n+m-2}}u_i^2+2w_i^{\frac{n+m+2}{n+m-2}}u_i\right)e^{-\phi_0}dV_{g_0}\\
        &+r^m_{\infty}\int_M\left( \frac{n+m-2}{n+m}\bar{w}_{z_i}^{\frac{2(n+m)}{n+m-2}}-\frac{n+m-2}{n+m}w_i^{\frac{2(n+m)}{n+m-2}} \right)e^{-\phi_0}dV_{g_0}.
    \end{split}
\end{equation*}
Using H\"older's inequality, we obtain
\begin{equation}\label{CS}
\begin{split}
    \int_M &w_i^{\frac{n+m+2}{n+m-2}}\left(R^m_{\phi}(t_i)-r^m_{\infty}\right)u_i e^{-\phi_0}dV_{g_0}\\
    &\leqslant \left(\int_M w(t_i)^{\frac{2(n+m)}{n+m-2}}|R^m_{\phi_i}-r^m_{\infty}|^{\frac{2(n+m)}{n+m+2}}e^{-\phi_0} dV_{g_0}\right)^{\frac{n+m+2}{2(n+m)}}\times\\
    &\left(\int_M |u_i|^{\frac{2(n+m)}{n+m-2}}e^{-\phi_0} dV_{g_0}\right)^{\frac{n+m-2}{2(n+m)}}.
    \end{split}
\end{equation}
Moreover, it follows from Corollary \ref{pro3} that
\begin{equation}\label{CS2}
    \begin{split}
        \int_M &\left(\frac{4(n+m-1)}{n+m-2}\left(L^m_{\phi_0}u_i,u_i\right)-\frac{n+m+2}{n+m-2}r^m_{\infty} \bar{w}_{z_i}^{\frac{4}{n+m-2}} u^2_i\right)e^{-\phi_0}dV_{g_0}\\
        &\geqslant c\int_M \left(\frac{4(n+m-1)}{n+m-2}L^m_{\phi_0}u_i,u_i\right)e^{-\phi_0}dV_{g_0}\\
        &\geqslant c Y_{n,m}\left(\int_M |u_i|^{\frac{2(n+m)}{n+m-2}}e^{-\phi_0} dV_{g_0}\right)^{\frac{n+m-2}{n+m}}.
        \end{split}
    \end{equation}
Finally, it follows from the pointwise estimate
\begin{equation*}
\begin{split}
     &\left|-\frac{n+m+2}{n+m-2}\bar{w}_{z_i}^{\frac{4}{n+m-2}}u_i^2+2w_i^{\frac{n+m+2}{n+m-2}}u_i+ \frac{n+m-2}{n+m}\left(\bar{w}_{z_i}^{\frac{2(n+m)}{n+m-2}}-w_i^{\frac{2(n+m)}{n+m-2}}\right) \right|\\
     &\leqslant C \bar{w}_{z_i}^{\max\{0,\frac{4}{n+m-2}-1\}}|u_i|^{\min\{\frac{2(n+m)}{n+m-2},3\}}+C|u_i|^{\frac{2(n+m)}{n+m-2}}
     \end{split}
\end{equation*}
that
\begin{equation}\label{CS3}
    \begin{split}
        \int_M &\left(-\frac{n+m+2}{n+m-2}\bar{w}_{z_i}^{\frac{4}{n+m-2}}u_i^2+2w_i^{\frac{n+m+2}{n+m-2}}u_i\right)e^{-\phi_0}dV_{g_0}\\
        &+\int_M\left( \frac{n+m-2}{n+m}\bar{w}_{z_i}^{\frac{2(n+m)}{n+m-2}}-\frac{n+m-2}{n+m}w_i^{\frac{2(n+m)}{n+m-2}} \right)e^{-\phi_0}dV_{g_0}\\
        & \leqslant C\int_M \left(\bar{w}_{z_i}^{\max\{0,\frac{4}{n+m-2}-1\}}|u_i|^{\min\{\frac{2(n+m)}{n+m-2},3\}}+|u_i|^{\frac{2(n+m)}{n+m-2}}\right)e^{-\phi_0}dV_{g_0}\\
        & \leqslant C \left(\int_M |u_i|^{\frac{2(n+m)}{n+m-2}}e^{-\phi_0} dV_{g_0}\right)^{\frac{n+m-2}{n+m}\min\{\frac{n+m}{n+m-2},\frac{3}{2}\}}\\
        & \leqslant o(1) \left(\int_M |u_i|^{\frac{2(n+m)}{n+m-2}}e^{-\phi_0} dV_{g_0}\right)^{\frac{n+m-2}{n+m}}.
    \end{split}
\end{equation}
Applying Cauchy--Schwarz inequality in (\ref{CS}) and combining this with (\ref{CS2}) and (\ref{CS3}), we conclude that 
\begin{displaymath}
\begin{split}
r^m_{\phi}(t_i)\leqslant &r^m_{\infty}+\left(E(\bar{w}_{z_i})-r^m_{\infty}\right)\left(\int_M \bar{w}_{z_i}^{\frac{2(n+m)}{n+m-2}} e^{-\phi_0}dV_{g_0}\right)^{\frac{n+m-2}{n+m}}\\
&+ C\left(\int_M w(t_i)^{\frac{2(n+m)}{n+m-2}}|R^m_{\phi_i}-r^m_{\infty}|^{\frac{2(n+m)}{n+m+2}}e^{-\phi_0} dV_{g_0}\right)^{\frac{n+m+2}{n+m}}\\
&\leqslant r^m_{\infty}+C\left(\int_M w(t_i)^{\frac{2(n+m)}{n+m-2}}|R^m_{\phi_i}-r^m_{\infty}|^{\frac{2(n+m)}{n+m+2}}e^{-\phi_0} dV_{g_0}\right)^{\frac{n+m+2}{2(n+m)}(1+\gamma)}.
\end{split}
\end{displaymath}
This completes the proof.
\end{proof}

\section{Nonpositive cases}
In this section, we deal with nonpositive cases; i.e. $Y_{n,m}[(g_0, \phi_0)]\leqslant 0$.
\subsection{Negative case}
As discussed in \cite[Proposition 3.5]{Case15}, we can choose an initial metric measure space $(M^n, g_0, e^{-\phi_0}dV_{g_0},m)$ such that $R^m_{\phi_0}<0$. Let $w(t)$ be the solution of (\ref{eq:factor}) on a maximal time interval $[0,T^*)$. Applying the maximal principle to (\ref{eq:factor}) derives
\begin{equation}\label{5.2}
    \frac{d }{d t}w_{\min}^{N}(t)\geqslant \frac{n+m+2}{4}\left( \min |R^m_{\phi_0}|w_{\min}(t)+r^m_{\phi}w_{\min}^{N}(t)\right),
\end{equation}
where $w_{\min}(t)=\displaystyle\min_{M}w(t)$ and $N=\frac{n+m+2}{n+m-2}$. By the constancy of volume, we have 
\begin{equation}\label{rY}
    r^m_{\phi(t)}\geqslant Y_{n,m}[(g_0, \phi_0)].
\end{equation}
By H\"older's inequality, we know that $Y_{n,m}[(g_0, \phi_0)]$ is finite. Hence, integrating (\ref{5.2}) yields
\begin{equation}\label{neglow}
    w_{\min}^{N-1}(t)\geqslant C\cdot\min \left\{w_{\min}^{N-1}(0), \frac{\min |R^m_{\phi_0}|}{|Y_{n,m}[(g_0, \phi_0)]|}\right\}.
\end{equation}
for a uniform constant $C$. On the other hand, the maximum principle also implies
\begin{equation}\label{5.5}
    \frac{d }{d t}w_{\max}^{N}(t)\leqslant \frac{n+m+2}{4}\left( -\left(\min R^m_{\phi_0}\right)w_{\max}(t)+r^m_{\phi}w_{\max}^{N}(t)\right),
\end{equation}
where $w_{\max}(t)=\displaystyle\max_{M}w(t)$. By (\ref{decreasing}), we conclude that
\begin{equation}\label{negsup}
    w_{\max}^{N}(t)\leqslant \left(w_{\max}^{N}(0)+1\right)e^{c(|\min R^m_{\phi_0}|+r^m_{\phi(0)})t},
\end{equation}
for some positive constant $c$. (\ref{neglow}) and (\ref{negsup}) imply that $w(t)$ will not blow up in finite time; i.e. $T^*=\infty$.

Next we claim that $r^m_{\phi(t)}$ will eventually become negative, even if it may not be so at the start. Indeed, if $r^m_{\phi(t)}$ is always nonnegative for $t\geqslant 0$, (\ref{5.2}) would imply
\begin{equation}
    \frac{d }{d t}w_{\min}^{N}(t)\geqslant \frac{n+m+2}{4} \min |R^m_{\phi^0}|w_{\min}(t).
\end{equation}
Hence $w_{\min}(t)$ approaches to infinity as $t\to \infty$. This contradicts the constancy of volume. Choosing a later time as the initial time, we may assume $r^m_{\phi(0)}<0$. (\ref{decreasing}) and (\ref{5.5}) yield
\begin{equation*}
    w_{\max}^{N-1}(t)\leqslant C\cdot \max\left\{w_{\max}^{N-1}(0), \frac{\max |R^m_{\phi^0}|}{|r^m_{\phi(0)}|}\right\}
\end{equation*}
for a uniform constant $C$. Together with (\ref{neglow}), we obtain that $w(t)$ is uniformly bounded from above and away from zero. 

Moreover, by (\ref{eq:R}), we obtain that
\begin{equation*}
\frac{d (R^m_{\phi})_{\min}}{d t}\geqslant (R^m_{\phi})_{\min}((R^m_{\phi})_{\min}-r^m_{\phi})\geqslant r^m_{\phi}((R^m_{\phi})_{\min}-r^m_{\phi}),
\end{equation*}
where $(R^m_{\phi})_{\min}(t)=\displaystyle\min_{M}R^m_{\phi}(t)$.
Combining this with (\ref{rY}), we can obtain a uniform lower bound on $ R^m_{\phi}(t)$; i.e. for all $t\geqslant 0$
\begin{equation}
   R^m_{\phi}(t)\geqslant r^m_{\phi}(t)-Ce^{r^m_{\phi(0)}t}\geqslant Y_{n,m}[(g_0, \phi_0)]-C.
\end{equation}
Similar to Proposition \ref{unilow}, the maximum principle also implies 
\begin{equation*}
     \sup_{M} R^m_{\phi}(t)\leqslant \max \left\{ \sup_{M}R^m_{\phi}(0), 0 \right\}.
\end{equation*}
Therefore, we can generalize Lemma \ref{eq: ev} and Proposition \ref{Holder} to negative case. In light of the foregoing argument at the end of Section 2, we derive uniform estimates for all higher order derivatives of $w(t)$, $t\geqslant 0$.

\subsection{Zero case}
In the final subsection, we treat the zero case. Without loss of generality, we can fix a background metric measure space $(M^n, g^0, e^{-\phi^0}dV_{g^0},m)$ such that $R^m_{\phi^0}\equiv 0$. Note that by \cite[Proposition 3.5]{Case15}, $r^m_{\phi(t)}$ can never be negative. Since the function $t\mapsto r^m_{\phi}(t)$ is nonincreasing, $r^m_{\phi(0)}=0$ implies $r^m_{\phi(t)}\equiv 0$. Thus the solution of (\ref{flow}) is constant in time. 

We next assume that $r^m_{\phi(0)}>0$. We observe that 
\begin{equation*}
    \frac{w_{\min}^{N}(t)}{w_{\min}^{N}(0)}\geqslant  c \int_{0}^t r^m_{\phi(t)} dt ~~\mbox{ and }~~ \frac{w_{\max}^{N}(t)}{w_{\max}^{N}(0)}\leqslant c \int_{0}^t r^m_{\phi(t)} dt
\end{equation*}
for some positive constant $c$. Hence we obtain the Harnack inequality
\begin{equation*}
     \frac{w_{\min}^{N}(t)}{w_{\min}^{N}(0)} \geqslant \frac{w_{\max}^{N}(t)}{w_{\max}^{N}(0)}.
\end{equation*}
It follows that $w(t)$ exists for all time.

By the same argument as Subsection 5.1, we can derive the smooth convergence.
\bibliography{mybib}{}

\begin{thebibliography}{Cas13b}

\bibitem[Bes08]{Besse87}
Arthur~L. Besse.
\newblock {\em Einstein manifolds}.
\newblock Classics in Mathematics. Springer-Verlag, Berlin, 2008.
\newblock Reprint of the 1987 edition.

\bibitem[Bre05]{Brendle05}
Simon Brendle.
\newblock Convergence of the {Y}amabe flow for arbitrary initial energy.
\newblock {\em J. Differential Geom.}, 69(2):217--278, 2005.

\bibitem[Bre07]{Brendle07}
Simon Brendle.
\newblock Convergence of the {Y}amabe flow in dimension 6 and higher.
\newblock {\em Invent. Math.}, 170(3):541--576, 2007.

\bibitem[Cas13a]{Case12}
Jeffrey~S. Case.
\newblock Conformal invariants measuring the best constants for
  {G}agliardo-{N}irenberg-{S}obolev inequalities.
\newblock {\em Calc. Var. Partial Differential Equations}, 48(3-4):507--526,
  2013.

\bibitem[Cas13b]{Case13}
Jeffrey~S. Case.
\newblock Sharp metric obstructions for quasi-{E}instein metrics.
\newblock {\em J. Geom. Phys.}, 64:12--30, 2013.

\bibitem[Cas15]{Case15}
Jeffrey~S. Case.
\newblock A {Y}amabe-type problem on smooth metric measure spaces.
\newblock {\em J. Differential Geom.}, 101(3):467--505, 2015.

\bibitem[Cas19]{Case19}
Jeffrey~S. Case.
\newblock The weighted {$\sigma_k$}-curvature of a smooth metric measure space.
\newblock {\em Pacific J. Math.}, 299(2):339--399, 2019.

\bibitem[Cho92]{Chow92}
Bennett Chow.
\newblock The {Y}amabe flow on locally conformally flat manifolds with positive
  {R}icci curvature.
\newblock {\em Comm. Pure Appl. Math.}, 45(8):1003--1014, 1992.

\bibitem[Eva10]{Evans10}
Lawrence~C. Evans.
\newblock {\em Partial differential equations}, volume~19 of {\em Graduate
  Studies in Mathematics}.
\newblock American Mathematical Society, Providence, RI, second edition, 2010.

\bibitem[Fri64]{Friedman13}
Avner Friedman.
\newblock {\em Partial differential equations of parabolic type}.
\newblock Prentice-Hall, Inc., Englewood Cliffs, N.J., 1964.

\bibitem[GT01]{DN01}
David Gilbarg and Neil~S. Trudinger.
\newblock {\em Elliptic partial differential equations of second order}.
\newblock Classics in Mathematics. Springer-Verlag, Berlin, 2001.
\newblock Reprint of the 1998 edition.

\bibitem[Ham88]{Hamilton88}
Richard~S. Hamilton.
\newblock The {R}icci flow on surfaces.
\newblock In {\em Mathematics and general relativity ({S}anta {C}ruz, {CA},
  1986)}, volume~71 of {\em Contemp. Math.}, pages 237--262. Amer. Math. Soc.,
  Providence, RI, 1988.

\bibitem[Lot07]{Lott07}
John Lott.
\newblock Remark about scalar curvature and {R}iemannian submersions.
\newblock {\em Proc. Amer. Math. Soc.}, 135(10):3375--3381, 2007.

\bibitem[Per02]{Perelman02}
Grigori Perelman.
\newblock The entropy formula for the ricci flow and its geometric
  applications.
\newblock {\em preprint}, 2002.

\bibitem[Sim83]{Simon83}
Leon Simon.
\newblock Asymptotics for a class of nonlinear evolution equations, with
  applications to geometric problems.
\newblock {\em Ann. of Math. (2)}, 118(3):525--571, 1983.

\bibitem[Ye94]{Ye94}
Rugang Ye.
\newblock Global existence and convergence of {Y}amabe flow.
\newblock {\em J. Differential Geom.}, 39(1):35--50, 1994.

\end{thebibliography}
\bibliographystyle{alpha}

\end{document}